\documentclass{amsart}

\usepackage{xcolor} 
\usepackage{graphicx} 
\usepackage{enumerate}

\usepackage{amsmath}
\usepackage{amssymb}
\usepackage{amsrefs}

\newcommand\axiom{\mathsf}

\newcommand\AP{\axiom{AP}}
 
\theoremstyle{plain}
\newtheorem{theorem}{Theorem}[section]
\newtheorem{lemma}[theorem]{Lemma}
\newtheorem{proposition}[theorem]{Proposition}
\newtheorem{corollary}[theorem]{Corollary}
\newtheorem{definition}[theorem]{Definition}

\theoremstyle{definition}

\newtheorem{claim}{Claim}

\theoremstyle{remark}

\numberwithin{equation}{section}

\DeclareMathOperator{\id}{id}
\DeclareMathOperator{\Succ}{Succ}

\DeclareMathOperator{\dom}{dom}

\DeclareMathOperator{\acc}{acc}

\DeclareMathOperator{\sakne}{stem}

\DeclareMathOperator{\Fn}{\mathop{Fn}}

\newcommand{\forces}[2]{\Vdash_{#1} \mbox{``\,} #2 \mbox{''}}

\newcommand{\Qposet}{{\mathbb Q}}
\newcommand{\Poset}{{\mathbb P}}

\newcommand{\cprec}{\mathop{<\!\!{\cdot}}}

\begin{document}
\title{Pseudo P-points and splitting number}
\author[A. Dow]{Alan Dow}
\address{Department of Mathematics,
University of North Carolina at Charlotte, 
Charlotte, NC 28223}
\email{adow@uncc.edu}
\author[S. Shelah]{Saharon Shelah}
\address{Department of Mathematics, Rutgers University, Hill Center,
 Piscataway, 
 New Jersey, U.S.A. 08854-8019}
\curraddr{Institute of Mathematics\\Hebrew University\\
  Givat Ram, Jerusalem 91904, Israel
}
\email{shelah@math.rutgers.edu}

\begin{abstract}
We construct a model in which the splitting number is large
and every ultrafilter has a small subset with no pseudo-intersection.
\end{abstract}

  \thanks{Research partially supported by NSF grant no: 136974.  Paper 1134
    on Shelah's list} 

\maketitle

Throughout the paper, Hyp$(\kappa,\lambda)$ will denote the
assumptions detailed in this paragraph. Each of $\kappa$ and $\lambda$
is a regular cardinal and
 $\aleph_1 < \kappa < \lambda$,
   $\lambda^{<\lambda} =\lambda$.
   The set $E$ is a stationary subset
   of   $  S^{\lambda}_\kappa$ 
   where $S^{\lambda}_\kappa\subset \lambda$ is
the set of
 ordinals of cofinality $\kappa$. There is a $\square$-sequence $\{
 C_\alpha : \alpha \in \lambda\}$ such that for limit 
ordinals $\alpha < \beta \in \lambda$ 
\begin{enumerate}
 \item $C_\alpha$ is a closed unbounded subset of $\alpha$,
 \item if $\alpha\in \acc(C_\beta)$, then $C_\alpha = C_\beta\cap\alpha$,
 \item $C_\alpha\cap E$ is empty.
\end{enumerate}
Naturally $E$ is a non-reflecting stationary set.  We also assume
there is a $\diamondsuit(E)$-sequence $\{ X_\alpha : \alpha \in E\}$,
 where $X_\alpha \subset \alpha$ and for all $X\subset \lambda$,
  there is a stationary set $E_X\subset E$ such that 
   $X_\alpha = X\cap \alpha$ for all $\alpha\in E_X$.

A set $b\subset \mathbb N$ is a pseudo-intersection of a family
 $\mathcal A$ of subsets of $\mathbb N$,  if $b$ is infinite
 and $b\setminus a$ is finite for all $a\in \mathcal A$. The
 pseudo-intersection 
 number of a free ultrafilter $\mathcal U$ on $\mathbb N$, denoted
  $\pi p(\mathcal U)$ is the smallest cardinal $\mu$
  such that there is a subset
   $\mathcal A\subset \mathcal U$ of cardinality $\mu$ with no
   pseudo-intersection.
   The splitting number, $\mathfrak s$, is very well known. It can be defined
   as the minimum cardinal such that for every  family $\mathcal A
   \subset [\mathbb N]^{\aleph_0}$ of smaller cardinality, there is a maximal free
 filter on the Boolean algebra generated by $\mathcal A\cup [\mathbb
 N]^{<\aleph_0}$  
 that has a pseudo-intersection. 
 
It was shown in \cite{BrendleShelah642} that it is consistent to have 
   $\pi p(\mathcal U)^+\leq \mathfrak s$ for all free ultrafilters
   $\mathcal U$ on  
   $\mathbb N$.  We construct a model by ccc forcing in which
    $\mathfrak s=\lambda=\mathfrak c$ and 
    $\pi p(\mathcal U)\leq \kappa$ for all ultrafilters $\mathcal U$ on $\mathbb N$.
    In fact we construct two models,  one in which $\mathfrak b = \lambda$
    and the second in which $\mathfrak b = \kappa$.

\section{preliminaries}

For a poset $(P,<_P)$, a set $D\subset P$ is  dense if
for each $p\in P$, there is a $d\in D$ with $p<d$.  Similarly, 
 a set $G\subset P$ is a filter (using the Jerusalem convention)
 if it is closed downwards and finitely
 directed upwards.  Therefore in the forcing language if $p<q$ are
 in $P$,  $q$ is a stronger condition and a subset $A$ of $P$ is
 an antichain if no pair of elements of $A$ have a common upper bound.
For convenience, we assume each forcing poset has a minimum element
$1_P$.  

A $P$-name $\dot a$ of a subset of $\omega$ (respectively $\mathbb N$)
 is {\textit{canonical} }
if for each $n\in \omega$ (respectively $n\in \mathbb N$),
 there is a (possibly empty) antichain $A_n$
of $P$ such that $\dot a = \bigcup \{ \{n\}\times A_n : n\in
\omega\}$.
There should be no risk of confusion if we abuse notation and let
each $n\in \omega$ also denote a $P$-name for itself.

For an infinite set $I$, 
the poset $\Fn(I,2)$ is the standard Cohen poset consisting of finite
partial functions  from $I$ into $2$ ordered by   extension. If
 $G$ is a generic filter for $\Fn(I\times \mathbb N ,2)$, 
 we obtain the standard
 sequence $\{ \dot x_i  : i\in I\}$  of canonical names
 for Cohen reals where 
 $\dot x_i = \{ (n, \langle (i,n),1\rangle) : n\in \mathbb N\}$. 
 We will refer to this  sequence as the canonical generic sequence we
 get from $\Fn(I\times \mathbb N, 2)$. 
 This family is forced to have the finite intersection property, 
 moreover, it is forced to be an independent family.  However, rather
 than design  a new poset
 we can use such  sequences to define a uncountable families
of  pairwise almost disjoint   subsets that are each Cohen
 over the ground model. Fix a sequence $\{ e_\alpha : \alpha\in \omega_1\}$
  (in the ground model) so that for each $\omega\leq 
  \alpha\in \omega_1$, $e_\alpha$ is a 
  bijection from $\omega$ onto $\alpha$.  
 
\begin{definition}
 For any\label{adfamily} sequence $\vec x = 
 \{ x_\alpha : \alpha \in\omega_1\}$ of subsets of $\mathbb N$,
 define, for 
 $\alpha\in \omega_1$, $c(\vec x,\alpha)$ and $a(\vec x, \alpha)$ where
\begin{enumerate}
 \item for $\alpha < \omega$, $c(\vec x,\alpha) 
 = x_\alpha\setminus \bigcup_{k<\alpha} x_k$, 
 \item for $\omega\leq \alpha$, $c(\vec x, \alpha) = 
  \{ \min\left( x_\alpha
   \setminus \bigcup\{ x_{e_\alpha(k)} : k < n\} \right): n\in \omega\}$,
   \item $a(\vec x, \alpha) = \mathbb N\setminus c(\vec x, \alpha)$.
\end{enumerate}
 \end{definition}

\begin{definition}
 For any set $I$, 
 we have the canonical generic sequence\label{cohsequence} 
   $\{  \dot x_{i,\alpha} : i\in I, \alpha\in \omega_1\}$ for the poset
    $\Fn(I\times \omega_1\times N,2)$.   For each $i\in I$, 
    we will let $\vec x_i $ denote the subsequence
       $\{  \dot x_{i,\alpha} : \alpha\in \omega_1\}$. 
       We then similarly have the sequences
         $\{ \dot c(\vec x_i, \alpha) : \alpha\in \omega_1\}$ and
           $\{ \dot a(\vec x_i, \alpha) : \alpha \in \omega_1\}$ defined
           as in Definition \ref{adfamily}.
\end{definition}

Let us recall that a poset $(P,<_P)$ is a complete suborder of a 
poset $(Q,<_Q)$ providing $P\subset Q$, $<_P~\subset~ <_Q$,
and each maximal antichain of $(P,<_P)$ is also
a maximal antichain of $(Q,<_Q)$. 
Note that it follows that incomparable members of $(P,<_P)$
are still incomparable in $(Q,<_Q)$, i.e. $p_1  \perp_P p_2$ 
implies $p_1 \perp_Q p_2$.
 We will say that a chain   $\{ P_i : i<\kappa\}$ of posets
  is a $\cprec$-chain of posets
 if $P_i\cprec P_j$ for all $i<j<\kappa$. We will say such a chain is
 a continuous 
 $\cprec$-chain if  $P_j = \bigcup_{i<j}P_i$ whenever $j$ has uncountable cofinality.
 We will use the term  strongly continuous for a chain $\{ P_\alpha :
 \alpha <\gamma\}$ 
 of posets 
 if $P_\beta =\bigcup_{\alpha<\beta }P_\alpha$ for all limits $\beta < \gamma$.

\begin{proposition}
 If $P\cprec Q$ and $q\in Q$, then there is a $p\in P$ (a projection)\label{projection} 
 with the property that for all $r\in P$ with $p<_P r$ ($r$ is
 stronger than $p$), 
  there is a $q_r\in Q$ that is stronger than each of $q$ and $r$. 
\end{proposition}
 
When we say that $V$ or $V'$ is a model, we will mean a transitive set
that is a model of a sufficiently large fragment of ZFC.

\begin{definition}
Let $V$ and $V'$ be models with $V\subset V'$.
\begin{enumerate}
\item If $P \in V$ and $Q\in V'$ are posets, we write $P<_V Q$ if
$P\subset Q$, $<_P\subset <_Q$, and 
each maximal antichain $A\subset P$ in $V$  is also a maximal antichain of $Q$.
Of course $P <_{V'} Q$ is the same as $V'\models P\cprec Q$.
\item A family $\mathcal A\subset [\mathbb N]^{\aleph_0}$ is thin over $V$ if
for each $\ell\in\omega$ and each infinite sequence $\{ H_n : n\in \omega\}\subset
[\mathbb N]^{\leq\ell}\cap V$ of pairwise disjoint sets, there is,
for each $a$ in the ideal generated by $\mathcal A$, an $n$ such that 
$H_n\cap a$ is empty. 
\item A family
 $\mathcal A\subset [\mathbb N]^{\aleph_0}$ is very thin over $V$
 if for each $a$ in the ideal generated by $\mathcal A$ and 
 each $g\in {\mathbb N}^{\mathbb N}\cap V$,
 there is an $n\in \mathbb N$ such that $a\cap [n,g(n)] $ is empty. 
\end{enumerate}
\end{definition}

We also will need the next result taken from 
 \cite{BrendleFischer}*{Lemma 13}. 
 
\begin{lemma} 
Let $\mathbb P, \mathbb Q$ be partial orders such that
 $\mathbb P\cprec   \mathbb Q$. Recall that
 the   name $\id_P = \{ (p,p) : p\in P\}$ is
 the $P$-name for the generic filter on $P$.
Let $\dot {\mathbb A}$ be a $\mathbb P$-name for a
 forcing notion\label{successorMatrix}
  and let $\dot {\mathbb B}$ be a $\mathbb Q$-name for a forcing
notion such that $\Vdash_{\mathbb Q} \dot {\mathbb A}  <_{V[\id_P]}
 \dot {\mathbb B}$,  
then 
$\mathbb P * \dot {\mathbb A} \cprec \mathbb Q * \dot  {\mathbb B}$  
\end{lemma}
  
  It is immediate that the conclusion of  Lemma \ref{successorMatrix} holds
  if $\mathbb P$ forces that
either $\mathbb A = \mathbb B$ or if $\mathbb A = \Fn(I,2) 
  \subset \Fn(J,2) = \mathbb B$.

\begin{proposition} 
 If $V\subset V'$ are models and $\mathcal A$ is thin (very thin) over
 $V$, then  for each $\alpha\in \omega_1$, $\mathcal A \cup \{ a(\vec
 x, \alpha)\}$ is thin  
 (respectively very thin)\label{staythin} 
 over $V$ where $\vec x = \{\dot x_\alpha : \alpha\in \omega_1\}$ is
 the canonical 
 generic sequence we get from forcing with $\Fn(\omega_1\times \mathbb
 N, 2)$ over 
  $V$.
 \end{proposition}

 \begin{proof}
   Fix any $a$ in the ideal generated by $\mathcal A$ and let
   $\{ H_n : n\in \omega\}\subset [\mathbb N]^{<\aleph_0}$ be any
   pairwise disjoint 
   family in $V$. Let 
   $p\in \Fn(\omega\times \mathbb N,2)$ be any condition and assume
   that $\{ n \in \omega : a\cap H_n \}$ is infinite. It suffices to
   prove that there is a $q$ extending $p$ and an $n$ such
   that $a\cap H_n=\emptyset$ and $q\Vdash H_n\cap \dot a(\vec
   x,\alpha)=\emptyset$.  We will skip the case when $\alpha<\omega$
   since it is easier.    Choose a finite set $F\subset \omega_1$ and
   an integer $m\in \mathbb N$ such that $\dom(p)\subset F\times
   \{1,\ldots, m\}$. By extending $F$ but not $m$, we can assume that
   $\{ e_\alpha(k) : k<m\}\subset F$ and $\dom(p) = F\times
    \{ 1,\ldots, m\}$. 
   Choose $n$ so that $m<\min(H_n)$ and
    $H_n\cap a=\emptyset$. We define an extension $q$ of $p$ that
    forces that $H_n\subset c(\vec x,\alpha)$.  
Let $\ell = \max(H_n)+m$ and $F' = F\cup \{ e_\alpha(k) : k \leq
\ell\}$.  Define $q\supset p$ so that for all
$(\beta,j) \in F'\times\{1,\ldots,\ell\}\setminus \dom(p)$,
$q(\beta,j) = 1$ if and only if $\beta=\alpha$. It is immediate
that $q\Vdash [m+1,\ell]\cap \dot x_{e_\alpha(k)} =\emptyset$ for all
$k<\ell$. Similarly, $q\Vdash [m+1,\ell]\subset \dot x_\alpha$. It
follows
that there is a $j_0\leq m$ such that $m+1$ is forced
by $q$ to be
the minimum element
of $\dot x_\alpha \setminus \bigcup\{\dot x_{ e_\alpha(k)} : k<j_0\}$.
Then, by induction on $1<j<\max(H_n)$, $m+j$ is
forced by $q$ to be the minimum element of
$\dot x_\alpha \setminus \bigcup\{\dot x_{ e_\alpha(k)} : k<j_0+j\}$. 

 \end{proof}

\begin{definition}
  For a poset $P$ and infinite set $X$,
  let $\wp(X,P)$ denote the set of canonical names
 of infinite subsets of $X$ (meaning $1_P$ forces that each
 $\dot a\in \wp 
 (X,P)$ is infinite). When $\mathcal E$ is a subset of
 $\wp(X,P)$ 
 we will use it in forcing statements   to mean the $P$-name
    $\{ (\dot a , 1_P) :   \dot a \in \mathcal E\}$.
\end{definition}

\begin{proposition} If $\{  P_i  : i < \kappa\} $ is a continuous\label{cohen}
$\cprec$-chain of ccc posets and if $\dot Q_i$ is the $P_i$-name  
for the poset $\Fn(i{+}1\times \theta \times \mathbb N, 2)$ for any
ordinal $\theta$,  
then $\langle P_i * \dot Q_i : i<\kappa\rangle$ is a continuous
$\cprec$-chain. 
If, in addition, $\{ \mathcal A_i : i<\kappa\}$ is a sequence such
that, for each $i<\kappa$, 
\begin{enumerate}
\item  
$\mathcal A_i\subset \wp(\mathbb N,P_{i+1})$,
\item $\mathcal A_i$ is forced
(by $P_{i+1}$) to be thin (respectively very thin) over 
the forcing extension by $P_i$,
\end{enumerate}
then, for each $i<\kappa$, $\mathcal A_i$ is forced to be thin (respectively
very thin) over
the forcing extension by $P_i*\Fn(i{+}1\times \theta\times \mathbb N, 2)$.
\end{proposition}

\section{The tools}
\newcommand{\az}{\axiom{a}}
\newcommand{\bz}{\axiom{b}}

\begin{definition}
$\AP$ is the    set of all structures $\axiom{a} \in H(\lambda)$, where\\
$\az
 = \langle \{ P^\az_i : i <\kappa\}, \{{\mathcal A}^\az_i : i<\kappa\}\rangle $
and for each\label{ap} $i<\kappa$
 
\begin{enumerate} 
 \item the sequence $\{ P^\az_i : i<\kappa\}$ is a continuous $\cprec$-chain of ccc posets,
 \item $\mathcal A^\az_i \subset \wp(\mathbb N,P^\az_{i+1})$,
 \item $P^\az_{i+1}$ forces that the ideal generated by $\mathcal A^\az_i$
 is thin over the forcing extension by $P^\az_i$.
\end{enumerate}
\end{definition}
 
\newcommand{\APv}{\axiom{APv}}

\begin{definition}
$\APv$ is the\label{apv}   set of all structures $\axiom{a} \in
H(\lambda)$, where\\ 
$\az
 = \langle \{ P^\az_i : i  < \kappa\}, \{{\mathcal A}^\az_i : i<\kappa\}\rangle $
and for each $i<\kappa$
 
\begin{enumerate} 
 \item the sequence $\{ P^\az_i : i < \kappa\}$ is a continuous
   $\cprec$-chain of ccc posets, 
 \item $\mathcal A^\az_i \subset \wp(\mathbb N,P^\az_{i+1})$,
 \item $P^\az_{i+1}$ forces that the ideal generated by $\mathcal A^\az_i$
 is very thin over the forcing extension by $P^\az_i$.
\end{enumerate}
\end{definition}

\newcommand{\lap}[1]{\leq^{#1}_{\AP}}

\begin{definition}
 For $i<\kappa$, we let $\lap i$ be the following two place relation on $\AP$:
 \( \az \lap i \bz \)   iff  for all 
 $ j\in [i,\kappa)$:  
 $\az,\bz \in \AP$,
  \( P^\az_j \cprec P^\bz_j \) ,
 and  \( \mathcal A^\az_j \subset \mathcal A^\bz_j\). 
 
 Similarly we let $\lap * = \bigcup_{i<\kappa} \lap i$, i.e. $\az\lap * \bz$ if  
  $\az\lap i \bz$ for some $i<\kappa$. 
  \end{definition}

For each $\az\in \AP$, we may let $P^{\az}_\kappa = \bigcup \{ P^{\az}_i : i<\kappa\}$
and note that $P^{\az}_i \cprec P^{\az}_\kappa$ for all $i<\kappa$. Similarly,
 it   follows immediately that 
 $P^{\az}_\kappa \cprec P^{\bz}_\kappa$
 whenever  $\az \lap * \bz$.

\begin{lemma}
 Each of $\lap i$ and $\lap *$ are transitive and reflexive orders on $\AP$.
 If $i<j$, then $\lap j \subset \lap i$.  If $\az\lap i \bz$ and $\bz\lap i\az$,
 then $P^\az_j = P^\bz_j$ and $\mathcal A^\az_j = \mathcal A^\bz_j$ for all
 $j\in [i,\kappa)$.
\end{lemma}

Since $\APv \subset \AP$,  we do not need new relation symbols
to denote the same binary relations on $\APv$.

\begin{lemma}
 If $\az \lap i \bz_1$ for some $i<\kappa$, then there is a $\bz_2\in \AP$
 such that $\az \lap 0 \bz_2$ and $\bz \lap i \bz_2$.  Similarly,
   if $\bz_1 \in \APv$, then $\az\in \APv$ and we can choose $\bz_2\in \APv$.
\end{lemma}

\begin{definition}
 For any $i<\kappa$ and
 ordinal $\delta$, a sequence $\langle \az_\alpha : \alpha < \delta\rangle$
 is a $\lap i$-increasing continuous chain if for all $\alpha < \beta <\delta$
 and $j\in [i,\kappa)$:
\begin{enumerate}
 \item $\az_\alpha \lap i \az_\beta$,
 \item the chain  $\{ P^{\az_\alpha}_j  : \alpha <\delta \}$ is strongly continuous, and,
  \item if $\alpha$ is a limit, then $\mathcal A^{\az_\alpha}_j = \bigcup
\{    \mathcal A^{\az_\xi}_j : \xi<\alpha\}$.
\end{enumerate}

 \end{definition}

\begin{lemma}
Suppose that  
  $\{\az_\alpha : \alpha < \delta\}$ is a $\lap *$-chain for some limit
  ordinal $\delta<\lambda$ and that there is a cub $C\subset\delta$
  and  an  $i<\kappa$  such that
    $\{ \az_\alpha : \alpha\in C\}$ is a
  $\lap i$-increasing continuous chain.  Then 
 there\label{ctschain} is an $\az_\delta\in \AP$ so that
 
\begin{enumerate}
\item  $\{ \az_\alpha : \alpha\in C\cup \{\delta\}\}  $ is also a $\lap i$-increasing
     continuous chain,
     \item 
 $\az_{\min(C)}\lap 0\az_\delta$, and
 \item   $\az_\alpha \lap * \az_\delta$ for all $\alpha < \delta$.
    \end{enumerate}
      If $i=0$, then $\{ \az_\alpha : \alpha \in C\}$ uniquely determines $\az_\delta$.
\end{lemma}

\begin{lemma} If
 $\{ \az_\alpha :\alpha  < \lambda\}$ is a $\lap *$-increasing chain
 from $\APv$
 and  if $\mathcal A^{\az_0}_i \neq \emptyset$ for all $i<\kappa$,
 then   the ccc forcing extension by\label{bkappa} 
   $P = \bigcup\{ P^{\az_\alpha}_\kappa:\alpha < \lambda\}$ satisfies that
   $\mathfrak b  \leq \kappa$.
\end{lemma}

\begin{proof}
 For each $i<\kappa$, choose any $\dot a_i\in \mathcal A^{\az_0}_i$
 and let $\dot f_i$ denote the order-preserving enumeration function
 from $\mathbb N$ onto
  $\dot a_i$. Note that $n\leq \dot f_i(n)$ for all $n\in \mathbb N$.
 Let $\dot g$ be any $P$-name of an element of 
   $\mathbb N^{\mathbb N}$.  
   Since $P$ is ccc, we can assume
   that $\dot g$ is a countable name. Choose any $\alpha\in \lambda$
   so that  
   $\dot g$ is a $P^{\az_\alpha}_\kappa$-name. Then 
    similarly choose $i_0<\kappa$ so that $\dot g$ 
   is a $P^{\az_\alpha}_{i_0}$-name. 
   Since $\az_0 \lap * \az_\alpha$ we may choose an $i>i_0$ so that
    $\az_0 \lap i \az_\alpha$. 
   Now we show that no condition $p\in P$   forces that $\dot f_i <^* \dot g$.
      Since $P^{\az}_{i+1} \cprec P$ and each of $\dot f_i$ and $\dot g$
   are $P^{\az}_{i+1}$-names, it suffices to prove that if $p\in P^{\az}_{i+1}$
   then, for any $n_0$ there is an extension $p'$ of $p$
   and an $n>n_0$ so that 
   $p'\forces{P^{\az_\alpha}_{i+1}} {\dot f_i(n) > \dot g(n)}$.
Since $\dot a_i\in \mathcal A^{\az_\alpha}_i$ and $\az_\alpha\in \APv$,
 there is a such a $p'$ and $n$ such that $p'\forces
  {P^{\az}_{i+1}}{\dot a_i \cap [n,\dot g(n)]=\emptyset}$. 
 There is no loss to assuming that $p'$ decides the value of the finite set
    $\{ k  < n : \dot f_i (k) < n \}$. If this set is empty, let $m=1$,
    otherwise, let $m$ be the maximum value.  Clearly $m < n$
    and we now have that $p'$ forces $\dot f_i(m+1) > \dot g(n)$.
    Since $\dot f_i$ is an increasing function, $p'\forces{P^{\az_\alpha}_i}{
       \dot f_i(n) > \dot g(n)}$ as required. 
\end{proof}

\begin{lemma}
 If $\{ \az_\alpha :\alpha  < \lambda\}$ is a $\lap *$-increasing chain
 then   the ccc forcing extension by\label{nopseudo} 
   $P = \bigcup\{ P^{\az_\alpha}_\kappa:\alpha < \lambda\}$ satisfies that
   if $\mathcal U\subset \wp(\mathbb N, P)$  is
   such that  $\{ i < \kappa : \mathcal U\cap \mathcal A_i\neq\emptyset\}$ has cardinality
    $\kappa$ for some $\alpha < \lambda$,
     then $\mathcal U$ does not have a pseudo-intersection.
\end{lemma}

\begin{proof}
Note that $P^{\az}_\kappa \cprec P$ for all $\alpha <\lambda$. 
 Let $\mathcal U\subset \wp(\mathbb N, P)$ and assume that 
 $\{ i< \kappa : \mathcal U\cap \mathcal A^{\az_\alpha}_i \neq \emptyset \}$ is 
 cofinal in $\kappa$.  Let $\dot b $ be any 
 canonical $P$-name of a subset of $\mathbb N$. 
 Choose $\alpha \leq 
 \beta<\lambda$ such that $\dot b$ is a $P^{\az_\beta}_\kappa$-name.
 Since $P^{\az}_\beta =\bigcup \{ P^{\az_\beta}_i : i<\kappa\}$, there is
 an $i_\beta < \kappa$ so that $\dot b$ is a $P^{\az_\beta}_{i_\beta}$-name.
 Choose $i<\kappa$ so that $\az_\alpha \lap i \az_\beta$. Now choose
 any $ j<\kappa$ so that $i,i_\beta < j $ and $\mathcal U\cap \mathcal 
 A^{\az_\alpha}_{j}$ is
not empty.   Since $\mathcal A^{\az_\alpha}_j\subset \mathcal A^{\az_\beta}_j$,
$\mathcal A^{\az_\alpha}_j$ is forced by $P^{\az_\beta}_{j+1}$ to be thin
over the forcing extension by $P^{\az_\beta}_j$. In particular, 
  $P^{\az_\beta}_j$ forces that 
  $\dot b$ is not   a subset of any element
 of $\mathcal U\cap \mathcal A^{\az_\alpha}_j$.  Since $P^{\az_\beta}_j\cprec
  P^{\az_\beta}_\kappa \cprec P$, this is also forced by $P$.
\end{proof}

By Proposition \ref{cohen} we can make the following definition.

\begin{definition}
 For any $\az\in \AP$ (or $\az\in \APv$) and 
 ordinal\label{cohenstep} 
  $\theta<\lambda$
 say that $\bz\in\AP$ is the
  Cohen$^\theta$-extension of $\az$ if, for
 each $i<\kappa$, 
\begin{enumerate}
 \item $P^{\bz}_i = P^{\az}_i *\Fn(i{+}1\times \theta\times \mathbb N, 2)$
 \item $\mathcal A^{\bz}_i = \mathcal A^{\az}_i$.
\end{enumerate}
\end{definition}

\begin{lemma}
 If $\az\in \AP$ (respectively $\az\in \APv$) and $\dot Q\in H(\lambda)$ is a 
  $P^{\az}_i$-name of a poset that is forced by $P^{\az}_\kappa$ to be 
   ccc, then\label{smallccc}
    there is a $\bz \in \AP$ (respectively $\bz\in \APv$) 
   such that $\az\lap 0 \bz$ and $P^{\bz}_\kappa = P^{\az}_\kappa*\dot Q$.
\end{lemma}

\begin{proof}
 We  define $\bz$ as follows.  Set $\mathcal A^{\bz}_j = \mathcal A^{\az}_j$
 for all $j<\kappa$. For $j<i$, let $P^{\bz}_j = P^{\az}_j$, and for
   $j\geq i$, let $P^{\bz}_j = P^{\az}_j * \dot Q$. By Proposition
   \ref{successorMatrix}, we have that $\{ P^{\bz}_j : j<\kappa\}$ is
    a continuous $\cprec$-chain.  By assumption, $P^{\bz}_j * \dot Q$ is ccc
    for all $j<\kappa$. Now we check that $\mathcal A^{\bz}_j$ is forced
    by $P^{\bz}_{j+1}$ to be thin (respectively very thin) over the forcing
    extension by $P^{\bz}_j$. For $j<i$ this is immediate. 
    
    Now assume that
     $i\leq j$ and that $\{ \dot H_n : n\in \omega\}$ is a sequence of
     $P^{\bz}_j$-names  that are forced to be pairwise disjoint 
     subsets of $[\mathbb N]^\ell$ (for some $\ell\in\omega$)
  and that $\dot g$ is a $P^{\bz}_j$-name of an
      element of $\mathbb N^{\mathbb N}$.  Let   $\dot a$
      be any name from $\mathcal A^{\az}_j$.
      Let $(p,q)$ be any
     condition in $ P^{\az}_{j+1}*\dot Q=P^{\bz}_{j+1} $. 
     We show that 
     $\bz\in \AP$ by showing that for some $n\in \omega$,
     $(p,q)$ has an extension forcing that $\dot H_n \cap \dot a$ is empty.
 We similarly show that if $\az\in \APv$, then  for some $n\in \omega$, 
  $(p,q)$ has an extension forcing that $[n,\dot g(n)]\cap \dot a$ is empty. 
 
   It will be convenient to pass to the forcing extension by $P^{\az}_{j+1}$
   so let $p\in G_{j+1}$ be a generic filter for $P^{\az}_{j+1}$.  Let $Q$
   denote the interpretation of $\dot Q$ by $G_j= G_{j+1}\cap P^{\az}_j$.
   We are now working in the extension $V[G_j]$. 
     Recursively
   define a sequence $\{ H_n' : n\in \mathbb N \} \subset [\mathbb N]^\ell$
   and   values $\{ m_n : n\in \mathbb N \}\subset\mathbb N$  
   so that, for each $n$, $\max(H_n') < \min(H_{n+1}')$, 
  and there is a
   condition $q_n\in Q$ stronger than $q$ such that, for some $r_n\in
    G_j$,  $(r_n,q_n) \forces{P^{\bz}_j}
     {H_n' = \dot H_n \ \mbox{and } \  \dot g(n) = m_n}$.
     Since 
    $\{ H_n' : n\in \omega\}$ is a pairwise disjoint sequence in 
      the forcing extension by $P^{\az}_j$, there is a $p'\in G_{j+1}$
      (stronger than $p$) and an $n\in \omega$ such that 
       $p'$ forces that $H_n'$ is disjoint from $a$. Since $G_{j+1}$
       is a filter, we may also assume that $p'$ is stronger than $r_n$.
       Now
       we have that $(p',q_n)$ is stronger than $(p,q)$ and
        $(p',q_n) \forces{P^{\bz}_{j+1}}{\dot H_n\cap \dot a = \emptyset}$.
    Similarly,  if $\az\in \APv$,
    there is an $n\in\mathbb N$  and a $p'\in G_{j+1}$ stronger
    than each of $p$ and $r_n$ such    that $p'$ forces
    that $\dot a$ is disjoint from $[n,m_n]$. This ensures
    that $(p',q_n)\forces{P^{\bz}_{j+1}}{[n,\dot g(n)]\cap \dot a=\emptyset}$. 
\end{proof}

Note that it follows from Lemma \ref{smallccc} that if $\az\in \AP$
and if $\{\dot Q_i : i<\kappa\}\in H(\lambda)$ is a sequence 
such that $\dot Q_i$ is a $P^{\az}_i$-name with $\{ 
 P_i *\dot Q_i : i<\kappa\}$ forming a continuous $\cprec$-sequence,
 then for each $i<\kappa$, $P_{i+1} *\dot Q_i$ forces that
  $\mathcal A^{\az}_i$ is thin over the extension by $P_i*\dot Q_i$. 
  This means that it is only the behavior of $\dot Q_{i+1}$ that affects
  if there is $\bz\in \AP$ with $\az\lap 0 \bz$ and $P^{\bz}_i
   = P^{\az}_i * \dot Q_i$ for all $i<\kappa$.

\begin{definition}
 If $\az\in \AP$ (respectively $\az\in \APv$) and $\dot Q\in H(\lambda)$ is
  a  $P^{\az}_\kappa$-name such that\label{cccstep}
  
\begin{enumerate}
 \item $\dot Q$ has cardinality less than $\kappa$,
 \item $P^{\az}_\kappa$ forces that $\dot Q$ is ccc
\end{enumerate}
then $\az * \dot Q$ denotes the $\lap 0$-extension $\bz$ 
as in Lemma \ref{smallccc} where $i<\kappa$ is chosen to be
minimal such
that $\dot Q$ is a $P^{\az}_i$-name. 
\end{definition}

This next lemma illustrates the device we use to ensure that every
 ultrafilter will have pseudo-intersection number at most $\kappa$.

\begin{lemma}
Suppose that $\delta<\lambda$ has cofinality $\kappa$ and that
  $\{ \az_\alpha : \alpha\in \delta\}\subset \AP$ is a $\lap *$-increasing sequence.
  Further suppose that there is a cub $C\subset \delta$
  of order type $\kappa$ such
  that $\{ \az_\alpha : \alpha \in C\}$ is a
    $\lap 0$-increasing
 continuous chain  and  that,  for each\label{APkappa}
 $\alpha\in \acc(C) $,
   $\az_{\alpha+1}$ is a Cohen$^{\omega_1}$-extension of $\az_\alpha$.
   Then, if  $P = \bigcup\{ P^{\az_\alpha}_\kappa : \alpha \in \acc(C) \}$ 
   and $\mathcal E \subset \wp(\mathbb N, P)$ is a maximal
   family that is forced to be a free ultrafilter
   on $\mathbb N$, there is a $\bz\in \AP$ such that
      $\az_\alpha \lap * \bz$ for all $\alpha\in \delta$, $P = P^{\bz}_\kappa$, 
      and, for all $i<\kappa$,
      $ \mathcal E\cap \mathcal A^{\bz}_i $ is not empty. 
\end{lemma}

\begin{proof}
 Let $\mathcal E\subset \wp(\mathbb N, P)$ and assume that 
 $1_P$ forces
 that $\mathcal E$ is a free ultrafilter on $\mathbb N$ and
 that $\mathcal E$ is a maximal such family.  This just means that if
  $\dot b\in \wp(\mathbb N, P)$ and $1_p\forces{P}{(\exists \dot e\in {\mathcal E}) 
     \dot b\supset \dot e}$, then $\dot b\in \mathcal E$.  
Let   $\{ \alpha_i : i<\kappa\}$ be the order-preserving enumeration
of $\acc(C)$.  For each $i<\kappa$ we now describe
how to choose a value $\beta_i \in \omega_1$. 
By
 our assumption, $\az_{\alpha_i+1}$ is a Cohen$^{\omega_1}$-extension. 
 That is,  $P^{\az_{\alpha_i+1}}_{i} $ is equal to 
   $P^{\az_{\alpha_i}}_{i+1} * \dot Q^{\az_{\alpha_i}}_{i+1}$
   where $\dot Q^{\az_{\alpha_i}}_{i+1}$ 
   is equal to (the trivial) $P^{\az_{\alpha_i}}_{i+1}$-name
   for 
$   \Fn(i+1\times \omega_1\times \mathbb N,2)$.
 Let $\vec x_{i}$ denote the canonical $\omega_1$-sequence associated
 with $\Fn(\{i\}\times \omega_1\times \mathbb N, 2)$ for this particular copy
 of Cohen forcing. Similarly, let $\{ c(\vec x_i, \beta), 
 a(\vec x_{i},\beta)  : \beta <\omega_1\} \subset 
 \wp(\mathbb N, P^{\az_{\alpha_i+1}}_{i+1})$ 
 be the family of names as constructed as in Definition \ref{adfamily}.
 Since the family $\{ c(\vec x_i , \beta) : \beta<\omega_1\}$ is forced to
 be pairwise almost disjoint, there is a maximal antichain $A_i\subset P$
 such that for each $p\in A_i$, there is a $\beta_p$ such that $p$
 forces that $a(\vec x_i, \xi)$ is in $\mathcal E$ for all $\beta_p<\xi \in \omega_1$.
 Since $P$ is ccc,  $A_i$ is countable, and so we may choose 
  any value $\beta_i \in \omega_1$ that is larger
 than $\beta_p$ for each $p\in A_i$. It follows
 that  $1\forces{P}{(\exists \dot e\in \mathcal E)
  a(\vec x_i,\beta_i)\supset \dot e}$. 
 By the maximality assumption on $\mathcal E$,  $a(\vec x_i,\beta_i)\in {\mathcal E}$.
 
 Now we define $\bz$. For each $i<\kappa$, $P^{\bz}_i = P^{\az_{\alpha_i}}_i$ and
  $\mathcal A^{\bz}_i = \mathcal A^{\az_{\alpha_i}}_i\cup \{ a(\vec x_i,\beta_i)\}$.
  Evidently we have that $\mathcal E\cap \mathcal A^{\bz}_i$ is not empty 
  for all $i<\kappa$.
  Since $P^{\az_{\alpha_i+1}}_{i+1} \cprec P^{\az_{\alpha_{i+1}}}_{i+1}$,
   we have that $\mathcal A^{\bz}_i$ is a subset of $\wp(\mathbb N, P^{\bz}_{i+1})$. 
   It follows from Proposition \ref{staythin} that $\mathcal A^{\bz}_i$ is forced
   to be thin over the forcing extension by $P^{\bz}_i = P^{\az_{\alpha_i}}_i$.
   Now for $i<j<\kappa$, $P^{\bz}_i = P^{\az_{\alpha_i}}_i \cprec
    P^{\az_{\alpha_i}}_j \cprec P^{\alpha_j}_j = P^{\bz}_j$.
Now suppose that   $j<\kappa$ is a limit of uncountable cofinality,
 we have to check that $P^{\bz}_j = \bigcup\{ P^{\bz}_i : i<j\}$. Let $p\in P^{\bz}_j$.
 Since $P^{\bz}_j = P^{\az_{\alpha_j}}_j = \bigcup \{ P^{\az_{\alpha_j}}_i : i<j\}$,
  we may choose $i_1<j$ such that $p\in P^{\az_{\alpha_j}}_{i_1}$. 
  By
  the assumption that $\{ \az_\alpha : \alpha \in C \}$ is a $\lap 0$-increasing
  continuous chain, there is an $i<j$ with $i_1 \leq i$ and $p\in 
    P^{\az_{\alpha_i}}_{i_1}$. Finally,
   $p\in P^{\az_{\alpha_i}}_{i_1} \subset P^{\az_{\alpha_i}}_i = P^{\bz}_i$ which
   completes this step.  It also shows that $P^{\bz}_\kappa =  P$.
     This completes
    the verification that $\bz\in \AP$. Fix any   $\xi <\delta$
    we verify that $\az_\xi\lap * b$.  Choose $i<\kappa$ so that $\xi<
    \alpha_i $ and choose   $i^*<\kappa$ so that 
     $\az_\xi \lap {i^* } \az_{\alpha_i}$.  We show that $\az_\xi \lap {i^*} \bz$.
    Let  $i^*\leq j<\kappa$. 
    First we have that   $\mathcal A^{\az_\xi}_j 
    \subset \mathcal A^{\alpha_i}_j\subset A^{\alpha_j}_j \subset A^{\bz}_j$. Secondly,
       $P^{\az_\xi}_j \cprec P^{\az_{\alpha_i}}_j \cprec P^{\az_{\alpha_j}}_j = P^{\bz}_j$. 
 \end{proof}

The proof of this next lemma is the same so the proof is omitted.

\begin{lemma}
Suppose that
  $\{ \az_\alpha : \alpha\in \kappa \}\subset \APv$ is a  $\lap 0$-increasing
 continuous chain for some $i<\kappa$ and  that,  for each limit\label{APvkappa}
 $\alpha\in \kappa$,
   $\az_{\alpha+1}$ is a Cohen$^{\omega_1}$-extension of $\az_\alpha$.
   Then, if  $P = \bigcup\{ P^{\az_\alpha}_\kappa : \alpha < \kappa\}$ 
   and $\mathcal E \subset \wp(\mathbb N, P)$ is a maximal
   family that is forced to be a free ultrafilter
   on $\mathbb N$, there is a $\bz\in \APv$ such that
      $\az_\alpha \lap * \bz$ for all $\alpha\in \kappa$, $P = P^{\bz}_\kappa$, 
      and 
      $\{ i < \kappa : \mathcal E\cap \mathcal A^{\bz}_i \neq \emptyset\}$
      has cardinality $\kappa$.    
\end{lemma}

\section{the Laver style posets}

In this section we develop the tools to allow us incorporate posets
into $\lap *$-chains that will increase the splitting number.
An ultrafilter $\mathcal D$ on $\mathbb N$ is Ramsey if for each
function $f$ with domain $\mathbb N$ and range an ordinal,
 there is a $D\in \mathcal D$
such that $f\restriction D$ is either constant or is strictly increasing.
For any family $\mathcal D$ of subsets of $ \mathbb N $
 that has the finite intersection property, we let $\langle \mathcal D\rangle$
 denote the filter generated by $\mathcal D$. We use the standard
 notation, ${\mathcal D}^+$, to denote the set of subsets of $\mathbb N$
 that meet every member of $\mathcal D$.

\begin{proposition}
If $\mathcal D_0$ is a free filter on $\mathbb N$ and $\theta$ 
is   an ordinal with $\theta\geq \mathfrak c$, then $\mathcal D_0$
can be extended to a   Ramsey ultrafilter
 in the forcing extension   by $\Fn(\theta,2)$. 
\end{proposition}

\begin{definition}
 For a  filter $\mathcal D$ on $\omega$, we define\label{deflaver}
 the Laver style poset
  $\mathbb L(\mathcal D)$ to be the set of trees $T\subset \mathbb N^{<\omega}$
  with the property that $T$ has a minimal branching node   $\sakne
(T)$
  and for all $\sakne(T)\subseteq t\in T$, the branching set
   $\Succ_T(t) = \{ k : t^\frown k \in T\}$ is an element of $\mathcal D$.
   For any tree $T\subset \mathbb N^{<\omega}$ and $t\in T$, we let
    $T_t = \{ s\in T : s\cup t \in T\}$.
 
The name $\dot L_{\mathcal D} = \{ (k, T) : (\exists t) 
 T =\left(\mathbb N^{<\omega}\right)_{t^\frown k}\}$ 
 will be referred to as the canonical name for the real
(pseudo-intersection)
added by $\mathbb L(\mathcal D)$.
\end{definition}

\begin{proposition}
 If $\mathcal D$ is any free filter on $\mathbb N$, then\label{sandb}
  $\dot L_{\mathcal D}$ is forced to be a pseudo-intersection for
   $\mathcal D$ and for every function $f\in \mathbb N^{\mathbb N}$,
   the enumeration function of $\dot L_{\mathcal D}$
    is  forced to be mod finite greater than $f$.
\end{proposition}

\begin{definition}
If $E$ is a dense subset of $\mathbb L(\mathcal D)$, then 
there is a (rank)
  function $\rho_E $ from $\mathbb N^{<\omega}$ into $\omega_1$ 
where 
  $\rho_E(t) = 0$ 
if and only if $t = \sakne
(T)$ for some $T\in E$,
and
 for all $t\in \mathbb N^{<\omega}$  and $0<\alpha\in \omega_1$, 
 $\rho_E(t) =\alpha$ if  $\alpha$ is  minimal such that the set\label{rank} 
 $\{ k \in \omega : \rho_E(t^\frown k)<\alpha\}$ is in 
 $\mathcal D^+$.  
\end{definition}

\begin{proposition}
 If $\mathcal D$ is a Ramsey ultrafilter and $E\subset \mathbb L(\mathcal D)$
 is a dense\label{rankR} set, then for each $t\in \mathbb N^{<\omega}$ with 
 $\rho_E(t) > 0$,
there is a    $D_t\in \mathcal D$ such that $\{ \rho_E(t^\frown k) : k\in D_t\}$
 is increasing and cofinal in $\rho_E(t)$. 
\end{proposition}

\begin{lemma}[\cite{charspectrum}*{1.9}]
 Suppose that
 $\Poset ,\Qposet $ are posets
with   $\Poset\cprec  \Qposet $.
Suppose also that  $\dot{\mathcal D}_0$ is 
 a $\mathbb P$-name of a  filter on $\mathbb N$ and
 $\dot{\mathcal D}_1$ is a\label{laver} 
 $\mathbb Q$-name of a  filter on $\mathbb N$.
  If $\Vdash_{\Qposet } \dot{\mathcal D}_0
\subseteq \dot{\mathcal D}_1 $
then $\Poset*\mathbb L(\dot{\mathcal D}_0)$  is a complete
subposet of $    \Qposet*\mathbb L(\dot{\mathcal D}_1)$
if either of the two equivalent conditions hold:
\begin{enumerate}
\item $\Vdash_{\Qposet} (\wp(\mathbb N, \Poset)\cap \dot{\mathcal
    D}_0^+)  \subseteq 
  \dot{\mathcal D}_1^+$, 
\item $\Vdash_{\Qposet}
\dot {\mathcal D}_1\cap \wp(\mathbb N, \Poset)
\subseteq \langle\dot{\mathcal D}_0\rangle$.
\end{enumerate}
\end{lemma}

\begin{proof}
Let $\dot E$ be any $\mathbb P$-name of a maximal antichain
of $\mathbb L(\dot{\mathcal D}_0)$. By Lemma \ref{successorMatrix},
it suffices to show that $\mathbb Q$ forces that every member
of $\mathbb L(\dot{\mathcal D}_1)$ is compatible with some member
of  $\dot E$. Let $G$ be any $\mathbb Q$-generic filter
and let $E$ denote the valuation of $\dot E$ by $G\cap \mathbb P$.
Working in the model $V[G\cap \mathbb P]$, we
  have the function $\rho_E $ as in Lemma \ref{rank}.
  Choose   $\delta\in \omega_1$ satisfying
 that $\rho_E(t) < \delta$ for all $t\in \omega^{<\omega}$.
 Now,
 working in $V[G]$, 
 we consider any $T\in \mathbb L(\dot{\mathcal D}_1)$ and we find
 an element of $E$ that is compatible with $T$.  In fact,  
 by induction on $\alpha <\delta$, one easily proves
 that for each $T\in \mathbb L(\dot{\mathcal D}_1)$ with
 $\rho_E(\sakne(T)) 
 \leq \alpha$, $T$ is compatible with some member of $E$.
\end{proof}

 If $\dot{\mathcal D_0}$ is the $\Poset$-name of a
maximal filter (ultrafilter), then the conditions in Lemma \ref{laver} hold.

\begin{lemma}
  If $V\subset V'$ are models and $\mathcal A\in V'$ is\label{Laverthin}
 thin over $V$, then
  for every Ramsey ultrafilter $\mathcal D\in V$, there is an
  ultrafilter $\mathcal D' \supset \mathcal D$ in $V'$ such that, for
  each $V'$-generic filter $G'$ for $\mathbb L(\mathcal D')$,
  $\mathcal A$ is thin over $V[G'\cap \mathbb L(\mathcal D)]$.  In
  other words, in the forcing extension of $V'$ by
  $\mathbb L(\mathcal D')$, $\mathcal A$ is thin over the forcing
  extension of $V$ by $\mathbb L(\mathcal D)$.
\end{lemma}

\begin{proof}
 Let $\mathcal O$ denote the set of strictly increasing functions 
    $f\in V$ such that $ f\in {\mathbb N}^D$ for some $D\in \mathcal D$.
     By the definition
    of thin over $V$, we may assume
    that $\mathcal A$ is closed under finite unions. 
 For each $D\in \mathcal D$, $a\in \mathcal A$ and  
    $f\in \mathcal O$, let $E(D,f,a) = 
    \{ n\in D\cap \dom(f) :  f(n) \notin a\}$.  
    We show that the family
      $\{ E(D,f,a) : f\in \mathcal O, D\in \mathcal D, a\in \mathcal A\}$ has the finite
      intersection property. It suffices to prove that if $\{ f_k : k< \ell \} $
      is  a finite subset of $\mathcal O$, $D\in \mathcal D$, and $a\in \mathcal A$, 
      then there is an $n\in D$ such that $f_k(n)\notin a$ for all $k<\ell$. 
      By shrinking $D$ we can assume that $D\subset\dom(f_k)$ for each $k<\ell$.
Choose any stictly increasing
 function $f\in V$ satisfying that for all $n\in \mathbb N$, 
   $[f(n),f (n+1))\cap D\neq\emptyset$, 
 and for all $j\in D$ with
  $j\leq f(n)$, $f_k(j) < f(n+1)$ for each $k<\ell$. Therefore, for each $n\in \mathbb N$ 
and   $j\in D\cap  [f(n), f(n+1))$, we have that 
   $f(n-1) < f_k(j) < f(n+1)$ for all $k<\ell$. 
   By re-indexing, we can assume that $\bigcup \{ [f(3n), f(3n+1)) : n\in \mathbb N\}$ 
   is in $\mathcal D$. Since $\mathcal D$ is Ramsey, we may choose $D_1\subset D$
   so that $ D_1 = \{ j_n : n\in \mathbb N\}$ and $f(3n)\leq j_n <   f(3n+1)$ for all $n\in
   \mathbb N$. Now define  $H_n = \{ f_k(j_n) : k < \ell\}$ and observe 
   that   $H_n \subset [f(3n-1), f(3n+2))$ 
   and so the sequence $\{ H_n :  n\in \mathbb N\}$ consists of pairwise disjoint
   sets. Since $\mathcal A$ is thin over $V$, there is an $n$ such
   that $H_n\cap a$ is empty. It follows that $j_n\in D$ and 
     $f_k(j_n) \notin a$ for each $k<\ell$ as required. 
     
     Let $\mathcal D'$ be any ultrafilter in $V'$ extending the family
     $\{ E(D,f,a) : f\in \mathcal O, D\in \mathcal D, a\in \mathcal
     A\}$.  Now we let $\{ \dot H_n : n \in \omega\}$ be a sequence in
     $V$ of $\mathbb L(\mathcal D)$-names that are forced by some
     $T_0 \in \mathbb L(\mathcal D)$ to be pairwise disjoint and of
     cardinality at most $\ell\in\omega$.  Let $a$ be any element of
     $\mathcal A$ and $ T_0' \in \mathbb L(\mathcal D')$ be any
     condition stronger than $T_0$.  We prove there is an extension
     $T_0'\supset T_1'\in \mathbb L(\mathcal D')$ and an $n\in \omega$
     such that $T_1'\Vdash \dot H_n\cap a$ is empty.  Let
     $t_0 = \sakne(T_0)$ and for each $1<n\in \omega$, let $H_{n,0}$
     be the maximal set such that there is a
     $T_n \in \mathbb L(\mathcal D)$ with
     $T_n\Vdash H_{n,0}\subset \dot H_n$ and $\sakne(T_n) = t_0$.
     There is a $D_0\in \mathcal D$ so that each element of the
     sequence $\{ H_{n,0} : n\in D_0\}$ has the same
     cardinality. Since we can assume that
     $D_0\subset \Succ_{T_0}(t_0)$, it follows that the elements of
     $\{ H_{n,0} : n \in D_0\}$ are pairwise disjoint. Choose any
     $1<n\in D_0$ so that $H_{n,0}\cap a $ is empty.  If
     $T_n \Vdash \dot H_n = H_{n,0}$, then we are done because $T_n$
     and $T_0'$ have the same stem, and so are compatible.  Let
     $\ell' \leq \ell$ be the value such that
     $T_n\Vdash | \dot H_n\setminus H_{n,0}| = \ell'$ and let
     $E_0 = \{ T\in \mathbb L(\mathcal D): \sakne (T)\notin T_n \
     \mbox{or} (\exists j) T \Vdash j\in \dot H_n\setminus H_{n,0}\}$.
     Since $E_0$ is a dense subset of $\mathbb L(\mathcal D)$, we have
     the associated rank function $\rho_{E_0}$ where for $t\in T_n$,
     $\rho_{E_0}(t) = 0$ implies that there is a $T\in E_0$ with
     $\sakne(T)=t$ and $j\in \mathbb N\setminus H_{n,0}$ such that
     $T\Vdash j\in \dot H_n$.  By the maximality assumption on $T_n$,
     we have that $\rho_{E_0}( t_0) > 0$. If $\rho_{E_0}(t_0)>1$, then
     by Proposition \ref{rankR}, there is a $k_0$ such that
     $1\leq \rho_{E_0}(t_0^\frown k_0) < \rho_{E_0}(t_0)$ and
     $t_0^\frown k_0 \in T_0'$. By repeating this step finitely many
     times, we can find a $t_1\in T_0'$ such that $\rho_{E_0}(t_1)$ is
     equal to $1$.  We may assume that $\rho_{E_0}(t_1^\frown k) = 0$
     for all $k\in \Succ_{T_0}(t_1)$.  For each
     $k\in \Succ_{T_0}(t_1)$, let $H_{n}(t_1^\frown k)$ be the maximal
     (non-empty) set of $j$ such that there is some condition in
     $\mathbb L(\mathcal D)$ with stem equal to $t_1^\frown k$ that
     forces $H_{n}(t_1^\frown k)\subset \dot H_n\setminus
     H_{n,0}$. There is some $D_{t_1}\in \mathcal D$ and
     $\ell_t\in\omega$ such that
     $\{ H_n(t_1^\frown k) : k\in D_{t_1}\}$ all have cardinality
     $\ell_t$.  For each $j <\ell_t$, define the function $f_j $ with
     domain $ D_{t_1} $ such that $f_j(k)$ is in $H_n(t_1^\frown k)$
     and, for each $k\in D_{t_1}$, $\{ f_j(k) : j<\ell_t\}$ enumerates
     $H_n(t_1^\frown k)$ in increasing order. By shrinking $D_{t_1}$
     we can assume that each $f_j\restriction D_{t_1}$ is either
     constant or is strictly increasing.  Since $\rho_{E_0}(t_1) > 0$,
     $f_0$ is not constant. To see this, assume that $f_0(k) = m$ for
     each $k\in D_{t_1}$. For each $k\in D_{t_1}$, choose a condition
     $T^k\in \mathbb L(\mathcal D)$ so that
     $\sakne(T^k) = t_1^\frown k$ and $T^k\Vdash m\in \dot H_n$. But
     now, the contradiction is that $\bigcup\{ T^k : k\in D_{t_1}\}$
     can be shown to be a condition with stem equal to $t_1$ that
     forces that $m\in \dot H_n$.  Choose any $k$ in the non-empty set
     $\Succ_{T_0'}(t_1) \cap \bigcap \{ E(D_{t_1}, f_j, a) :
     j<\ell'\}$ and set $t_2 = t_1^\frown k\in T_0'$.  Now define
     $H_{n,1} = H_{n,0}\cup H_n(t_1^\frown k)$ and choose
     $T_2\subset T_0$ as above so that $\sakne(T_2) = t_2$ and
     $T_2\Vdash H_{n,1}\subset\dot H_n$.  Define $E_2$ analogous to
     how we defined $E_0$ so that for $t\in (T^k)_{t_2}$,
     $\rho_{E_2}(t) = 0$ if and only if there is a condition with stem
     $t$ that forces some $j$ to be in $\dot H_n\setminus H_{n,1}$.
     We again note that when we proved that each $f_j$ ($j<\ell'$)
     above was strictly increasing, we have also shown that
     $\rho_{E_2}(t_2)>0$.  Since $H_{n,1}$ is a proper extension of
     $H_{n,0}$ and is also disjoint from $a$, we can repeat this
     argument finitely many (at most $\ell$) times until we have found
     an element $t\in T_0'$ which has a stem preserving extension that
     forces $\dot H_n$ is disjoint from $a$.
\end{proof}

\begin{corollary}
 Suppose that $\az \in \AP$ and let\label{laverstep}
   $|P^{\az}_\kappa|\leq \theta=\theta^{\aleph_0} < \lambda$,
   then there is a $\bz\in \AP$ and a sequence
  $\{ {\mathcal D_i} : i<\kappa\}$ satisfy that, for $i<j<\kappa$,
 \begin{enumerate}
\item $\az \lap 0\bz$,
 \item $  {\mathcal D_i}\subset \wp(~\mathbb N, P_i*
\Fn(i{+}1\times\theta,2)~)$,
 \item ${\mathcal D_i}\subset {\mathcal D_j}$,
 \item   
    $P_i*\Fn(i{+}1\times\theta,2)$ forces that $\mathcal D_i$ 
    is   a Ramsey ultrafilter   on $\mathbb N$, 
\item 
    $P^{\bz}_i = P^{\az}_i * \Fn(i\times\theta,2)*\mathbb L({\mathcal D_i  })$.
   \end{enumerate}
\end{corollary}

\begin{definition}
 For any $\az\in \AP$ we will say\label{laverdef}
  that $\bz$ is an $\mathbb L(\vec{\mathcal D})$-extension
 of $\az$ if there is a sequence $\{ \mathcal D_i : i<\kappa\}$ such that for all $i<j<\kappa$, 
 
\begin{enumerate}
\item  $\az \lap 0 \bz$,
 \item $\mathcal D_i \subset \wp(\mathbb N,P^{\az}_i)$,
 \item $\mathcal D_i$ is forced by $P^{\az}_i$ to be a Ramsey ultrafilter
 on $\mathbb N$, 
 \item $\mathcal D_i \subset \mathcal D_j$, 
 \item $P^{\bz}_i = P^{\az}_i * \mathbb L(\mathcal D_i)$.
\end{enumerate}
\end{definition}

\section{$\pi p(\mathfrak U) \leq \kappa $ and 
$\mathfrak b =   \mathfrak s = \lambda$} 

Fix a 1-to-1 function $h$ from $\lambda$ onto
 $H(\lambda)$.
Recall that $\{ X_\alpha : \alpha \in E\}$ is the
 $\diamondsuit$-sequence on $\lambda$ as
 in Hyp$(\kappa,\lambda)$.

\begin{theorem}
Assume Hyp$(\kappa,\lambda)$.
 There\label{mainAP} is a sequence $\{ \az_\alpha , \zeta_\alpha :
 \alpha\in \lambda\}$ 
 such that for each limit $\delta\in \lambda$
 
\begin{enumerate}
\item the sequence $\{ \az_\alpha :\alpha < \delta\}$ is  $\lap
  *$-increasing, 
  $\{ \zeta_\alpha : \alpha <\delta\}$ is non-decreasing,
and   $\zeta_\delta \in \lambda$ is the supremum of $\{ \zeta_\alpha :
\alpha <\delta\}$, 
 \item if $\delta\notin E$, 
 the sequence $\{ \az_\alpha : \alpha\in \acc(C_\delta) \cup
 \{\delta\}\}$ is   a
    $\lap 0$-increasing continuous chain,
    \item if $\delta\in E$ and\label{iultra}
     $\mathcal E_\delta = \{ h(\xi) : \xi\in X_\delta\}$ is
    a maximal 
   subset of $\wp(\mathbb N,  P^{\az_\delta}_\kappa ) $ that is forced
   by $P^{\az_\delta}_\kappa$ to be a free
   ultrafilter on $\mathbb N$, then $\mathcal E_\delta\cap 
    \mathcal A^{\az_\delta}_i$ is not empty for all $i<\kappa$,
    \item $\az_{\delta+1}$ is the Cohen$^{\omega_1}$-extension of $\az_{\delta}$
    and $\zeta_{\delta+1} = \zeta_\delta$,

    \item if $\alpha = \delta+1$ then
     $\zeta_{\alpha+1} = \zeta_{\alpha}$ 
   and $\az_{\alpha+1} $ is the
      Cohen$^{\theta_\alpha}$-extension of $\az_{\alpha}$ where
   $\theta_\alpha= |P^{\az_{\alpha}}_\kappa|^{\aleph_0}$
 \item if $\alpha = \delta+2$, then $\zeta_{\alpha+1}=\zeta_\alpha$
      and $\az_{\alpha+1}$ is an\label{ilaver} 
      $\mathbb L(\mathcal D)$-extension of $\az_{\alpha}$,
      \item if $\alpha \in (\delta+2,\delta+\omega)$, then 
       $\zeta_{\alpha+1}$ is the minimal value strictly above $\zeta_\alpha$
       such that $\dot Q_{\alpha+1} = h(\zeta_{\alpha+1}-1)$
      has cardinality less than $\kappa$ and\label{iccc}   is a 
         $P^{\az_{\alpha}}_\kappa$-name of a poset that is forced to be ccc,
         and 
      $\az_{\alpha+1} = \az_{\alpha}*\dot Q_{\alpha+1}$ as in Definition
       \ref{cccstep}. 

\end{enumerate}
\end{theorem}

\begin{proof}
 The proof is by induction on limit $\delta<\lambda$. We can define
  $\az_0$ so that $P^{\az_0}_i = \Fn(i\times \mathbb N, 2)$ 
  and $\mathcal A^{\az_0}_i = \emptyset$ for all
    $i<\kappa$.  Similarly, for $n\in \omega$, let $\az_{n+1}$ be
    the Cohen$^\omega$-extension of $\az_n$.  For all $n\in \omega$,
     set $\zeta_n = 0$.
     If $\delta$ is a limit
    ordinal not in $E$ and $\acc(C_\delta)$ is cofinal in $\delta$,
     then  $\{ \az_\alpha : \alpha \in \acc(C_\delta)\}$ is a $\lap 0$-increasing
     continuous chain, and so $\az_\delta$ is defined as in
     Proposition \ref{ctschain}. 
     If $\acc(C_\delta)$ is not cofinal in $\delta$, then let $\alpha_0$ be the 
     maximum element of $\acc(C_\delta)$, and let $C= \{ \alpha_n : n\in \omega\}$
     enumerate $C_\delta\setminus \alpha_0$. There is an $i<\kappa$
 such that $\{ \az_{\alpha_n} : n\in \omega\}$ is a $\lap i$-increasing continuous
 chain.  Again applying Proposition \ref{ctschain} produces $\az_\delta$ so
 that $\az_{\alpha_0}\lap 0\az_{\delta}$. Case (5) is handled
 by Proposition \ref{cohen} per Definition \ref{cohenstep}.  Similarly,
  per Definition \ref{laverdef} and by inductive assumption (5), Corollary
  \ref{laverstep} handles inductive step (6).  Inductive step (7)
  is handled by Definition \ref{cccstep} and Lemma \ref{smallccc}.
  
  Now we consider inductive step (3) when $\delta\in E$. By induction
  hypothesis (1), we have that $\{  \az_\alpha : \alpha \in \acc(C_\delta)\}$
  is a $\lap 0$-increasing chain. Let $P$ denote the poset
    $\bigcup \{ P^{\az_\alpha}_\kappa : \alpha\in \acc(C_\delta)\}$. 
    We recall from Lemma \ref{nopseudo}
    that $P^{\az_\beta}_\kappa\cprec P$ for all $\beta <\delta$.
   If 
  $\{ h(\xi) : \xi \in X_\delta\} $ is a subset of $\wp(\mathbb N, P)
  $ and is forced to have the finite intersection property, 
  we can use Zorn's Lemma to enlarge it to a maximal such family.
  Otherwise,  choose $\mathcal E_\delta\subset \wp(\mathbb N,P)$
   to be any maximal family which is forced to have the 
   finite intersection property.
   Since $\wp(\mathbb N,P)$ consists only of sets that are forced to be
   infinite, $\mathcal E_\delta$ is forced to be a free filter.  We
   prove that $\mathcal E_\delta$ is forced to be an ultrafilter.
   Assume that 
   $\dot b$ is any canonical $P$-name of a subset of $\mathbb N$
   and that $p\in P$ is any condition that forces $\dot b$ meets
   every member of $\mathcal E_\delta$. 
There is a $\dot e\in \wp(\mathbb N, P)$
   such that $p\forces{P}{\dot b=\dot e}$ and, for all $q\in P$
   that are incomparable with $p$, $q\forces{P}{\dot e = \mathbb N}$. 
   Clearly then $1_P$ forces that $\mathcal E_\delta\cup \{ \dot e\}$
   has the finite intersection property.  By the maximality of $\mathcal E_\delta$,
     $\dot e\in \mathcal E_\delta$. This proves that 
     $p\Vdash \dot b\in \mathcal E_\delta$.   
     Now apply Lemma \ref{APkappa}.
     
 This completes the proof. 
\end{proof}

\begin{theorem} 
  Assume Hyp$(\kappa,\lambda)$. There is a ccc poset $P$
  forcing\label{b=s=lambda}  
  that $\mathfrak s = \mathfrak b = \lambda$, $MA(\kappa)$, and
  $\pi p(\mathcal U) \leq \kappa$ for all free ultrafilters
  $\mathcal U$ on $\mathbb N$.
\end{theorem}

\begin{proof}
 Let $\{ \az_\alpha , \zeta_\alpha : \alpha \in \lambda\}$ be the sequence 
 constructed in Theorem \ref{mainAP}. Let $P$ be the poset
  $\bigcup \{ P^{\az_\alpha}_\kappa : \alpha \in \lambda\}$. Since
  $\{ P^{\az_\alpha}_\kappa : \alpha \in \lambda\}$ is a strongly continuous
   $\cprec$-increasing chain of ccc posets, it follows that $P$ is ccc.
   Furthermore $\wp(\mathbb N , P)$ is equal to the union of the
   increasing sequence $\{\wp(\mathbb N, P^{\az_\alpha}_\kappa)
  :\alpha < \lambda\} $. It then follows immediately from condition (\ref{ilaver})
  that $\mathfrak s$ is forced to be $\lambda$. Similarly 
  by Proposition \ref{sandb}, $P$ forces that   $\mathfrak b = \lambda$. 
 Now we check that $P$ forces that $MA(\kappa)$ holds: that is,
  if $Q$ is any $P$-name of a ccc poset of cardinality less than $\kappa$
  and $\{ A_\xi : \xi < \mu\}$ is a family of maximal antichains of $Q$
  with $\mu<\kappa$, then there is a filter $G$ on $Q$ that meets each
   $A_\xi$. To show that this is forced by $P$, we may choose a
  $P$-name $\dot Q$ for $Q$ as well as $P$-name $\{ \dot A_\xi : \xi<\kappa\}$
  for the maximal antichains. We may assume that $1_P$ forces
  that $\dot Q$ is ccc.  Since $\dot Q$ has cardinality less than $\lambda$,
      there is an $\alpha < \lambda$ such that $\dot Q$, and each $\dot A_\xi$
        is a $P^{\az_\alpha}_\kappa$-name.  Choose any $\gamma<\lambda$
        so that $\dot Q*\Fn((\gamma,\gamma+\omega),2)$ is not in the list
          $\{ h(\zeta) : \zeta \leq \zeta_\alpha\}$ and let $h(\zeta') = \dot Q*\Fn((\gamma,
          \gamma+\omega),2)$. It follows from inductive conditions (1)
          and (\ref{iccc}), that the sequence $\{ \zeta_{\alpha} : \alpha\in \lambda\}$ 
          is unbounded in $\lambda$. So we may choose limit $\delta < \lambda$
          maximal so that $\zeta_\delta \leq \zeta'$. Since $\zeta_{\delta+3}=\zeta_\delta$
          and  
          $\zeta_{\delta+\omega}$ 
          is greater than $\zeta'$, there is a minimal $n\geq 3$ such that
            $\zeta'<\zeta_{\delta+n+1}$. 
            It should be clear that by inductive condition (\ref{iccc})
             that $\zeta'+1$ is equal to $\zeta_{\delta+n+1}$ and so
            $\dot Q_{\alpha_{\delta+n+1}}$ was chosen to be 
            $\dot Q*\Fn((\gamma,\gamma+\omega),2)$.  This ensures
            that $P^{\az_{\delta+n+1}}_\kappa$ forces that there is a filter
on $\dot Q$  that meets each $\dot A_\xi$. 

Now let $\dot {\mathcal  U}$ be a $P$-name of an ultrafilter on $\mathbb N$.
As we did in the inductive step (\ref{iultra}), we can let $\mathcal E$
be the set of all $\dot e\in \wp(\mathbb N, P)$ that are forced by 
 $1_P$ to be an element of $\dot {\mathcal U}$. 
   There is a cub $C\subset\lambda$
 such that for all $\delta\in S^\lambda_\kappa$, $\mathcal E\cap
  \wp(\mathbb N, P^{\az_\delta}_\kappa)$ is a maximal set
  that   is   forced by
   ${P^{\az_\delta}_\kappa}$ to be an ultrafilter on $\mathbb N$. 
   Now let $X = \{ \xi \in \lambda : h(\xi)\in \mathcal E\}$.
 We can pass to a cub subset $C'$ of $C$ so that for
 all $\delta\in C'\cap S^\lambda_\kappa$, 
 the set $\{ h(\xi) : \xi\in X\cap \delta\}$ is equal to 
     $\mathcal E\cap \wp(\mathbb N, P^{\az_\delta}_\kappa)$.   
     Now choose a $\delta\in E\cap C'$
      so that $X_\delta = X\cap \delta$.       It follows
      from inductive step (\ref{iultra}), that $\mathcal E\cap \mathcal
       A^{\az_\delta}_i$ is not empty for all $i<\kappa$.         
       By Lemma \ref{nopseudo},  $P$ forces that $\dot {\mathcal U}$
       has a subset of size $\kappa$ with no pseudo-intersection.
\end{proof}

\section{another ccc poset for raising $\mathfrak s$}

The proper 
poset $\mathcal Q_{Bould}$ is introduced in \cite{Boulder}
(also Sh:207 in the Shelah archive)
to establish the consistency of $\mathfrak b < \mathfrak s = \mathfrak
a$.  For special directed subfamilies $\mathcal D$ of $\mathcal
Q_{Bould}$,  there is a ccc poset denoted $Q(\mathcal D)$ that is
analogous to $\mathbb L(\mathcal E)$ for filters $\mathcal E$ on
$\mathbb N$ (see Definition \ref{bigQ}).
Let us note the important properties of $\mathcal Q_{Bould}$ shown
to hold in \cite{Boulder}. 
 The first is that it adds an unsplit real.

\begin{proposition}
If $\dot L$ is the generic subset of $\mathbb N$ added by\label{3.1}
 $\mathcal  Q_{Bould}$, then the set
 $\{ A \subset \mathbb N : A \in V \mbox{\ and\ } 
| \dot L\setminus A | < \aleph_0 \}$ is a free ultrafilter over
 $V\cap \mathcal P(\mathbb N)$.
\end{proposition}

The second is that the forcing does not add a dominating real. By
Lemma \ref{bkappa}, this property is needed if such a
$Q(\mathcal D)$ is to replace $\mathbb L(\mathcal E)$ in
constructing $\az$ in $\APv$.

\begin{proposition}
If $\dot f$ is a $\mathcal Q_{Bould}$-name of a
function in $\mathbb N ^\mathbb N$ and if $\dot L$ is the generic subset of
 $\mathbb N$ added by $\mathcal Q_{Bould}$, then there is a 
(ground model)
 $h\in \mathbb N^\mathbb N$ so that, for every infinite set
 $A\subset \mathbb N$ in the ground model, 
the  set $\{ n \in A : \dot f(n) < h(n)\}$ will
be forced to be infinite.
\end{proposition}

We adopt the elegant representation of this poset from
 \cite{Abraham}. Also 
many of the technical details 
for constructing ccc subposets 
 of this poset,
sharing the above mentioned properties,
are similar to the results in   \cite{FischerSteprans}.
The main tool is to utilize logarithmic measures.

\begin{definition}
A function $h$ is a logarithmic measure on a set
 $S\subset \mathbb N $ if $h$ is a function from $[S]^{<\aleph_0}$ into
 $\omega$ with the property that whenever $\ell\geq 0$ and
$h(a\cup b) \geq \ell+1$, then either
 $h(a)\geq\ell$ or $h(b)\geq\ell$. A pair $(s,h)\in \mathcal L_n$ 
if $s\in [\mathbb N]^{<\aleph_0}$ and $h$ is a logarithmic measure on $s$
 with $h(s) \geq n$.  The elements $e$ 
of $[\mathbb N]^{<\aleph_0}$ such
that $h(e)>0$ are called the positive sets. 
\end{definition}

\newcommand{\innt}{\mathop{int}}

When we discuss $t\in \mathcal L_1$, we use $\innt(t)$ and
$h_t$ to denote the pair where $t = (\innt(t) , h_t)$.
We say that a subset $e$ of $\innt(t) $ is $t$-positive
to mean that $h_t(e) > 0$. 
Note that if $(s,h)\in \mathcal L_n$ and $\emptyset \neq e\subset s$, 
then $(e,h\restriction [e]^{<\aleph_0})\in \mathcal L_{h(e)}$.

\begin{definition} The poset $\mathcal Q_{Bould}$ consists of
all pairs $(u,T)$ where
\begin{enumerate}
\item $u\in [\mathbb N]^{<\aleph_0}$,
\item $T = \{ t_\ell : \ell \in \omega\}$  is a sequence of 
members of $\mathcal L_1$ where for each $\ell$,
$\max(\innt(t_\ell)) < \min(\innt(t_{\ell+1}))$, and 
the sequence $\{  h_{t_\ell}(\innt(t_\ell)) : \ell \in \omega\}$
is monotone increasing and unbounded.
\end{enumerate}
 For each $(u,T)\in \mathcal Q_{Bould}$,
 let $\ell_{u,T}$ be the minimal $\ell$ 
 such that $\max(u)<  \min(\mathop{int}(t_\ell))$ and
 let   $\mathop{int}(u,T) = \bigcup \{ \mathop{int}(t_\ell) : 
\ell_{u,T}\leq  \ell \}$.

For $T_1 =  \{  t^1_\ell : \ell \in \omega\} $
and    $T_2 = \{   t^2_\ell : \ell \in \omega\}$ 
 with $(u_1,T_1), (u_2,T_2)\in \mathcal Q_{Bould}$,
the extension relation is
 defined by 
$(u_2, T_2)    \geq
 (u_1,  T_1) $ (stronger) 
 providing 
\begin{enumerate}
\item $u_2\supset u_1$ and $u_2\setminus u_1 $ is contained in 
     $\mathop{int}(u_1,T_1)$,
\item  $\mathop{int}(u_2,T_2) \subset \mathop{int}(u_1,T_1)$ 
\item there is a sequence of finite subsets of $\omega$,
 $\langle B_k : k\in \omega\rangle$, such that
for each $k \geq \ell_{u_2,T_2}$,
 $\max(B_k) <  \min(B_{k+1})$  and
 $\mathop{int}(t^2_k) \subset \bigcup \{\mathop{int}(t^1_\ell ) : \ell
 \in B_k\}$,
\item for every $k\geq\ell_{u_2,T_2}$ 
and every $t^2_k$-positive $e\subset
  \mathop{int}(t^2_k )$ there is a $j\in B_k$ such
that $e\cap \mathop{int}(t^1_j)$  is $t^1_j$-positive.
\end{enumerate}
For a finite subset $\mathcal D$ of $\mathcal Q_{Bould}$ and an
element   $t$ of $\mathcal L_1$, we say that $t$ is built from
$\mathcal D$ if there is a $q= (\emptyset,T_q)\in \mathcal Q_{Bould}$
with  $t\in T_q$ such that $q\geq (\emptyset, T_{q'})$
 for each $q'=(u_{q'},T_{q'})\in \mathcal D$.
\end{definition}

\newcommand{\Pure}{{\mathcal P}_{Bould}}

\begin{definition}
 The elements $q \in \mathcal Q_{Bould}$ of the form $(\emptyset, T_q)$
 are called pure conditions.  We let $\Pure$ denote
 the set of all pure conditions in $\mathcal Q_{Bould}$.
 A family $\mathcal D\subset \Pure$ is
 finitely compatible if each finite subset of $\mathcal D$ has an
 upper  bound  in $\mathcal Q_{Bould}$.  The family $\mathcal D$ is finitely
 directed if each finite subset has  an upper bound in $\mathcal D$.
\end{definition}

For an element $q \in \mathcal Q_{Bould}$,
we let $u_q$ and $T_q$ denote the elements with $q = (u_q,T_q)$.
We also use $\innt(q)$ for $\innt(u_q,T_q)$. 
The fact
that the elements of $T_q$ are enumerated by $\omega$ is unimportant.
 It will be convenient to adopt the convention that for an infinite set
  $L\subset \omega$, 
 the
  sequence $(\emptyset , \{ t_\ell : \ell\in L\})$ is a pure condition
so long as $(\emptyset , \{ t'_n : n\in \omega\})\in\mathcal Q_{Bould}$
 where $\{ \ell_n : n\in \omega\}$ is the increasing enumeration of $L$
 and $t_n' = t_{\ell_n}$ for each $n\in \omega$.

\begin{definition}
If $\mathcal D$ is a finitely directed set of pure conditions,
 then define $Q(\mathcal D)$ to be the subposet
  $\{ (u,T)  : u\in [\mathbb N]^{<\aleph_0}, (\emptyset , T)
  \in \mathcal D\}$\label{bigQ}  of $\mathcal Q_{Bould}$.
  We let  $\dot L_{\mathcal D}$ denote the $Q(\mathcal D)$-name
  $\{ (n,q) : q\in Q(\mathcal D)\ ,\ n\in u_q\}$.
\end{definition}

\begin{proposition} If $\mathcal D\subset \Pure$ is  finitely
  directed, then  $Q(\mathcal D) $  is\label{centered}
  a   $\sigma$-centered poset. Each $q\in Q(\mathcal D)$ forces
  that $\dot L_{\mathcal D}\setminus \innt(T_q)\subset u_q$.
\end{proposition}

It follows from the results in 
\cite{FischerSteprans} 
 that there is a  $\Fn(2^\omega,2)$-name
$\dot{\mathcal D}$ that is  forced to be a
 finitely directed subset of 
 $\Pure$ with the property that $\Fn(2^\omega,2)*Q(\dot{\mathcal
   D})$ will add an unsplit real and not add a dominating real
 (see Lemma \ref{ultrasplit} and Lemma \ref{getready}).
These
 will be the  factor   posets we will use in
 place of $\mathbb L(\mathcal D)$ in the construction of members of
 $\APv$.
 We will also need an analogue of Lemma \ref{laver} and we now
introduce
 a condition on $\mathcal D$ that will ensure
 that $Q(\mathcal D)  <_V Q(\mathcal D_2)$ when $\mathcal D_2\supset
 \mathcal D$  in a forcing extension of $V$
 (see Proposition \ref{q207}).

\begin{definition}
Let $\mathcal D\subset \Pure$ be  directed mod finite. 
For a set
$E=  \{ (u_n, T_n) : n\in \omega\}\subset  Q(\mathcal D)$
we  say that 
$(\emptyset, \{  t_\ell :  \ell\in \omega\})\in\Pure$, 
is a mod finite meet of $E$ if,  
for each $0<\ell\in \omega$, $w\subset \max(\innt(t_{\ell-1}))$, 
 and   $h_\ell$-positive $e\subset 
 \mathop{int}(t_\ell)$, there is an $n <\min{\innt(t_{\ell+1})}
 $ and a $w_e\subset e$
such that
$(w\cup w_e ,\{t_m: m>\ell\})\geq (u_n, T_k)$ for each $  k\leq \max\{n
,\ell,\max(\innt(t_{\ell-1}))\}$.

A set of pure conditions
$\mathcal D$ is $\aleph_1$-directed mod finite\label{direct}
if it is directed mod finite and each predense subset
of $Q(\mathcal D)$ has a mod finite meet in $\mathcal D$.
\end{definition}

\begin{lemma}
Suppose that   $\{(u_n,T_n) : n \in \omega\}$ is a
 subset   of $\mathcal Q_{Bould}$
 and let $T$ be a mod finite\label{whymeet}
 meet. 
 Then, for each $w\in [\mathbb N]^{<\aleph_0 }$, 
     $\{ (u_n, T_n ) : n\in \omega\}$ is 
  predense below   
  $(w,T)   $ in all of $\mathcal Q_{Bould}$.
\end{lemma}

\begin{proof} 
  Let $(w,T')$ be an arbitrary member of $\mathcal Q_{Bould}$
  that is  compatible with $(w,T)= (w,\{t_\ell : \ell \in \omega\})
$ in $\mathcal Q_{Bould}$.
  By extending $(w,T')$, we may assume that $(w,T')
   <(w,\{ t_\ell : \ell>\ell_{w}\})$, where
 $w\subset \min(\mathop{int}(t_{\ell_w}))$.
  Choose any $T'$-positive $e$
   so that $(w\cup e, T') < (w,T')$ and $\max(w)  < \min(e)$.
   Therefore there is an $\ell >\ell_w$  
   such 
   that $h_\ell(e\cap \mathop{int}(t_\ell)) >0$,
and so, by Definition \ref{direct},
 there is an $n<\min(\mathop{int}(t_{\ell+1}))$ 
and a $w_e\subset e$
so that 
    $(w\cup w_e,T) < (u_n,T_n)$ and we have
   that $(w\cup w_e,T')<(w,T) < (u_n,T_n)$. 
\end{proof}

\begin{definition} 
 $\mathbb Q_{207}$ is  the set of $Q(\mathcal D)$ 
 where $\mathcal D$ is an $\aleph_1$-directed mod finite
 set of pure conditions. 
\end{definition}

\begin{proposition}
Assume%
\label{q207}
   $\{ P_i : i<\kappa \}$ is a continuous $\cprec$-chain 
   of ccc posets  and that 
$\{  \dot{  Q}_{i} : i < \kappa\}$  is a
chain such that, for each $i<\kappa$, $\dot {Q}_i$
is a 
 $  P_{ i}$-name of   a member of $\mathbb Q_{207}$,
  then $\{ P_i * \dot Q_i : i<\kappa\}$ is a 
  continuous $\cprec$-chain of ccc posets. 
 \end{proposition}

\begin{proof}
  By Lemma \ref{successorMatrix}, it suffices to prove
  that each   $  P_{j}$-name 
  of a predense subset of $\dot {  Q}_{ j}$
  is forced by $ P_{ i}$ to be predense
  in $\dot { Q}_{ i}$.  Since
   $\dot {Q}_{ j}$ is forced to be 
   a subset of $\dot { Q}_{ i}$, 
   it is immediate that $[\mathbb N]^{<\aleph_0}\times
    \{T\}$ is a predense subset of 
      $\dot { Q}_{ i}$ for each
       $(\emptyset,T)\in \dot{ Q}_{ j}$.
   Now the  Proposition follows by Lemma \ref{whymeet}.
\end{proof}

\begin{definition} Say that a subset $\tilde {\mathcal L}$ of
  $\mathcal L_1$ is\label{Dpositive}
 $\mathcal D$-positive if for each finite
  $\mathcal D'\subset \mathcal D$ and each $n\in \omega$,
  there is a $t\in \tilde {\mathcal L}\cap \mathcal L_n$ 
  that is built from $\mathcal D'$.
\end{definition}

\begin{proposition} If $\tilde{\mathcal L}\subset \mathcal L_1$ is
  $\mathcal D$-positive for some $\mathcal D\subset\Pure$, then for
  each\label{directpositive}
 finite ${\mathcal D}'\subset \mathcal D$,
  the set $\{ t\in \tilde {\mathcal L} : t \ \mbox{is built from}\
{  \mathcal D}'\}$ is $\mathcal D$-positive.
\end{proposition}

The poset $\Fn(\mathbb N,2)$ is forcing isomorphic to the poset
$\omega^{<\omega}$ ordered by extension. Similarly,
each  infinite branching (non-empty) 
 subset $S\subset\omega^{<\omega}$ is forcing isomorphic to
 $\omega^{<\omega}$; we say that $S$ is infinite branching
 if,  for each $s\in S$, the 
$\{n \in\omega : s^{\frown} n\in  S\}$ is infinite.
For such infinite branching  $S\subset\omega^{<\omega}$ 
 and each $k\in \omega$,
let $\dot n_k^S$ denote the $S$-name
 $\{ (s(k) , s) : s\in S\ \mbox{and}\ k\in \dom(s)\}$. 

\begin{definition}
Fix an enumerating  function $\lambda$ from $\omega$ 
onto $ \mathcal L_1$.  For $\mathcal D\subset \Pure$, 
say that $S\subset \omega^{<\omega}$ is
   ${\mathcal D}^+$-branching if $\emptyset\in S$ and, for each $s\in
S$,\label{Dbranch} 
\begin{enumerate}
\item for each $k\in\dom(s)$, $\lambda(s(k))\in \mathcal L_{k}$,
\item  $\max(\innt(\lambda(s(j)))) < \min(\innt(\lambda(s(k))))$
  for $j<k\in\dom(s)$, and
\item   the set $\{ \lambda(n) : s^{\frown} n\in S\}$ is a
  $\mathcal D$-positive set~.
\end{enumerate}
For each $k\in \omega$, 
define the $S$-name $\dot r^S_k $ to be $\lambda(\dot n^S_k)$. For
each finite ${\mathcal D}'\subset \Pure$,
let $\dot I^S_{{\mathcal D}'}$ be the $S$-name for the set
 $\{ k\in \omega : \dot r^S_k \ \mbox{is built from \ } {\mathcal D}'\}$.
\end{definition}

\begin{lemma} If $\mathcal D\subset \Pure$ is finitely compatible
  and
 if $S\subset\omega^{<\omega}$ is ${\mathcal    D}^+$-branching
 then\label{makedirected} 
$ \Vdash_{S} \dot{\mathcal D}^S = \mathcal D \cup
  \{ (\emptyset , \{\dot r^S_k : k\in \dot I^S_{{\mathcal D}'}\} )
  : {\mathcal D}'\in [\mathcal D]^{<\aleph_0}\}$
  is finitely directed. 
\end{lemma}

\begin{lemma}
  If $\mathcal D\subset \Pure$ is finitely directed,\label{ultrasplit}  
  then there is a $\Fn(\mathbb N,2)$-name $\dot {\mathcal D}_1$ such
  that \begin{enumerate}
\item $\Vdash_{\Fn(\omega,2)}  \mathcal D\subset
  \dot{\mathcal D}_1\subset \Pure\ \ \mbox{and}\ \dot{\mathcal D}_1$
 is finitely directed,
\item for each $A\subset \mathbb N$,
  $\Vdash_{\Fn(\mathbb N,2)} ~(\exists q\in \dot{\mathcal D}_1)~
  (   \innt(q)\subset A \ \ \mbox{or}\ \ \innt(q)\cap A=\emptyset)$
  \end{enumerate}
\end{lemma}

\begin{proof}
Let $S\subset \omega^{<\omega}$ be the maximal ${\mathcal
  D}^+$-branching set.  That is, $S$ is the set of all 
$s\in \omega^{<\omega}$ that satisfy  properties (1) and (2) of
Definition \ref{Dbranch}. For finite subsets ${\mathcal D}'$ of
$\mathcal D$, let $\dot I^S_{{\mathcal D}'}$ be an $S$-name for
the set $\{ k \in \omega : \dot r^S_k
\ \mbox{is built from}\  {\mathcal D}'\}$.

 For a subset
  $A$ of $\mathbb N$, define
   $\mathcal L(A)$ to be 
  $\{ t\in \mathcal L_1 : \innt(t)\subset  A\}$.
  If $\mathcal L(A)$ is not $\mathcal D$-positive, there is a finite
  ${\mathcal D}_A\subset \mathcal D$ and an integer $m$ 
such that there is no
$t\in \mathcal L_{m}\cap \mathcal L(A)$
 that is built from  ${\mathcal D}_A$.
Since, for each $t\in \mathcal L_{n+1}$, there is an $e\subset
\innt(t)$
such that $h_t(e) \geq n$  and either $e\subset A$
or $e\subset (\mathbb N\setminus A)$, it follows
that  if $\mathcal A$ is  a finite
  partition of $\mathbb N$, then $\mathcal L(A)$ is $\mathcal
  D$-positive for some $A\in \mathcal A$.  Therefore, by Zorn's Lemma,
  there is  a free ultrafilter $\mathcal U$ on $\mathbb N$
  so that $\mathcal L(U)$ is $\mathcal D$-positive for all $U\in
  \mathcal U$.   For each $U\in \mathcal U$, let 
  $\dot I^S_{U}$ denote an
$S$-name that will evaluate to
  $\{ k\in \omega : \innt(\dot r^S_k) \subset U\}$. 
It
  follows  from the fact
  that  $\mathcal L(U) $ is $\mathcal D$-positive, that
  $\Vdash_S \dot I^S_U \cap \dot I^S_{{\mathcal D}'}$ is infinite
for each ${\mathcal D}'\in[\mathcal D]^{<\aleph_0}$.
 It is also clear
  that $\Vdash_S \{ \dot I^S_U \cap
\dot I^S_{{\mathcal D}'}: U\in \mathcal U\}$ is closed under
  finite intersections. It then follows that
$ \Vdash_{S} \dot{\mathcal D}^S = \mathcal D \cup
  \{ (\emptyset , \{\dot r^S_k : k\in \dot I^S_{{\mathcal D}'}\} )
  : {\mathcal D}'\in [\mathcal D]^{<\aleph_0}\}$
is the desired finitely directed subset of $\Pure$. 
\end{proof}

In order to produce extensions of finitely directed
$\mathcal D\subset \Pure$ that are $\aleph_1$-directed mod
finite, we will need the following tools for   constructing members
of $\mathcal L_n$ for arbitrarily large $n$.
A family $L\subset [\mathbb N]^{<\aleph_0}$  naturally induces a
logarithmic measure. 

\begin{definition}
Let $L\subset [\mathbb N]^{<\aleph_0}$ and define\label{Pmeas}
 the relation $h(s) \geq \ell$ for $s\in [\mathbb N]^{<\aleph_0}$
 by induction
 on $|s|$ and $\ell$ as follows:
\begin{enumerate}
\item $h(e) \geq 0$ for all $e\in [\mathbb N ]^{<\aleph_0}$,
\item $h(e) > 0$ if $e$ contains some non-empty element  of $L$,
\item for $\ell >0$, $h(e)\geq \ell+1$ if and only if,
 $|e|>1$ and
  whenever $e_1,e_2\subset e$ are such that
$e = e_1\cup e_2$ then $h(e_1)\geq \ell$ or $h(e_2)\geq \ell$. 
\end{enumerate}
The definition of $h(e)$ is the maximum $\ell$ such 
that $h(e)\geq \ell$. 
\end{definition}

\begin{proposition}[\cite{Abraham}*{Lemma 4.7}]
Let $L\subset [\mathbb N]^{<\aleph_0}$ be\label{afterPmeas}
 an upward closed family of
non-empty sets and let $h$ be the associated logarithmic measure.
 Assume that whenever 
$\mathbb N$ is partitioned into finitely many sets $\mathcal A$,
 there is some $A\in \mathcal A$ such that $L\cap [A]^{<\aleph_0}$ is
 non-empty. Then, for any partition $\mathcal A$ of $\mathbb N$,
 and any integer $n$, there is an $A\in \mathcal A$ and an
 $e\subset A$ such that $h(e)\geq n$.
\end{proposition}

\begin{lemma}
If $\mathcal D \subset \Pure$
 is  finitely directed and  $E = \{ (u_n, T_n) : n\in \omega\}$
is a  subset of $Q(\mathcal D)$\label{fuse},
then in the forcing extension by $\Fn(\mathbb N,2)$,
 there is a finitely directed 
 $\mathcal D\subset \mathcal D_1 \subset \Pure $ 
 such that either $E$ is not predense in $Q(\mathcal D_1)$
 or there is a
 condition $(\emptyset, T)   \in \mathcal D_1$ such that 
$(\emptyset,T)$  is the mod finite meet of $E$.
\end{lemma}

\begin{proof}
  We assume that in the forcing extension by $\Fn(\mathbb N,2)$, 
  $E$ is a predense subset of $Q(\mathcal D_1)$ for each
  finitely directed $\mathcal D_1$ with $\mathcal D\subset
  \mathcal D_1\subset \Pure$. We will prove there is a
  ${\mathcal D}^+$-branching $S\subset \omega^{<\omega}$ satisfying
  that $\Vdash_{S} (\emptyset, \{\dot r^S_k : k\in \omega\})$ is
  a mod finite meet of $E$. By  Lemma \ref{makedirected},
  $S$ forces that there is a $\mathcal D_1$ as required. Since $S$ will
  be   forcing isomorphic to $\Fn(\mathbb N,2)$, this will complete the
  proof.

Let 
$L
 = \{ w\in [\mathbb N]^{<\aleph_0} : 
(\exists n\in \omega) ~~
u_n \subset w  \}~$, and for
 each $w\in [\mathbb N]^{<\aleph_0}$,
 let $$L_w =
 \{ w_1 \in [\mathbb N]^{<\aleph_0} :
\max(w)<\min(w_1)\ \mbox{and}\ 
 w\cup w_1 \in L \}.$$
Say that $t\in \mathcal L_1$ is $(E,\ell)$-large (for $\ell\in
\mathbb N$) if $\ell<\min(\innt(t))$ and
for each $w\subset\{1,\ldots, \ell\}$ each $t$-positive 
set contains an element of $L_w$.  For each $(E,\ell)$-large $t$,
let $N_t\geq \max(\innt(t))$
 denote a sufficiently large integer such that each
$t$-positive set contains an element of $\{ u_n : n < N_t\}$. 

\begin{claim}
  For\label{large}
 each $\ell\in\mathbb N$, the
  set $\{ t\in \mathcal L_\ell : t \ \mbox{is}\
   (E,\ell)\mbox{-large}\} $ is $\mathcal D$-positive.
\end{claim}

\bgroup
\def\proofname{Proof of Claim \ref{large}.}

\begin{proof}
  Let  $\ell\in \mathbb N$.
  Since $\mathcal D$
  is finitely directed, in order to show
  that the set of $(E,\ell)$-large
  elements is $\mathcal D$-positive, it suffices
  to show that for each  $q\in \mathcal D$
  there is a $(E,\ell)$-large $t$ that is built from
  $q$. 
Note $t\in \mathcal L_1$ is built from
  $q$ if $\innt(t)\subset \innt(q)$ and,
  for each $t$-positive set $e$, there is a 
$t_e\in T_q$ 
  such that $e\cap \innt(t_e)$ is $t_e$-positive.
  Say that a finite set $e$ is $q$-positive if
  there is a $t_e\in T_q$ such
  that $e\cap \innt(t_e)$ is $t_e$-positive.
  Note also that if $q_1,q\in \Pure$ and $q_1\geq q$,
then each $q_1$-positive set is also $q$-positive.

Define  $L_{q,\ell}$ to be the elements of 
$\bigcap_{ w\subset \{1,\ldots, \ell\}}   L_w $
that are also $q$-positive. Let $h_{q,\ell}$ denote the associated
logarithmic measure as in Definition \ref{Pmeas}. If $e\in
L_{q,\ell}$ is a subset of $\innt(q)$, then
 $(e, h_{q,\ell}\restriction [e]^{<\aleph_0})$ is built from
$q$. Therefore,
to  finish the proof
of the claim it will suffice to prove that there is an $e\in
L_{q,\ell}$ with $h_{q,\ell}(e) > \ell$. 
We prove this using Proposition
\ref{afterPmeas}; so let $\mathcal A$ be a finite partition of
$\mathbb N$. Pass to the forcing extension by $\Fn(\mathbb N,2)$
and choose, by Corollary \ref{ultrasplit}, a finitely directed
$\mathcal D_1\subset \Pure$ that contains $\mathcal D$
and satisfies that there is a $q_1\geq q$ in $Q(\mathcal D_1)$
and an $A\in \mathcal A$
such that $\innt(q_1)\subset A$. We may arrange  that
$\ell < \min(\innt(q_1))$. 
By assumption, $E$ is a predense
subset of $Q(\mathcal D_1)$. Let $ \{ t_k : k\in\omega\}$ be
the standard enumeration of $T_{q_1}$. 
For each $w\subset \{1,\ldots,\ell\}$,
there is a $q_w \geq (w,T_{q_1})$ and an $n_w\in\omega$ 
such that $q_w \geq (u_{n_w}, T_{n_w})$.
There is a $K_w\in  \omega$ such that
$(u_{q_w}\setminus w) \subset \bigcup\{ \innt(t_k) : k\in K_w\}$.
Since $u_{q_w}\supset u_{n_w}$, we have that any finite set containing
$u_{q_w}$ is in $L_w$.  This shows that, for some $K\in\omega$,
$e = \bigcup_{k<K} \innt(t_k) $ is in $L_{w}$ for
each $w\subset \{ 1, \ldots, \ell\}$. Since $e$ is $q_1$-positive
and $q_1\geq q$, it follows that $e$ is $q$-positive.  This completes
the proof that $e\in L_{q,\ell}$. 
\end{proof}
\egroup

For each $s\in \omega^{<\omega}$,
define $\ell_s$ to be the maximum element of the set
$\{1\}\cup \bigcup_{k\in\dom(s)}  \innt(\lambda(s(k)))$.
Now define the infinite branching $S\subset\omega^{<\omega}$ by the
recursive rule that, for each $s\in S$ 
  and $n\in\omega$,
$s^{\frown }n \in S$ if and only if
\begin{enumerate}
  \item 
$\lambda(n) $ is $(E,\ell_s)$-large, and
\item $\lambda(n)$ is built from $\{ (\emptyset, T_n ) : n
  < \max \{ N_{s(k)}: k\in \dom(s)\} \}$
  and $(\emptyset , T_{\ell_s})$.
\end{enumerate}
Claim 1  and Proposition \ref{directpositive}
show that $S$ is ${\mathcal D}^+$-branching. 
The definition of the notion of being $(E,\ell)$-large and the second
criterion of being an element of $S$, ensures that
 $\Vdash_S (\emptyset, \{ \dot r^S_k : k\in \omega\})$ is the mod
finite meet of $E$. Since $\Fn(\mathbb N,2)$ is forcing isomorphic to
$S$,
 the proof of the Lemma now follows from Lemma \ref{makedirected}.
\end{proof}

The next result follows by 
first applying Lemma \ref{makedirected}
to obtain  directed mod finite  extension of $\mathcal D$,
next  applying Lemma 
\ref{ultrasplit}, 
and finally   repeatedly applying Lemma \ref{fuse}
in a recursive construction of length $\mathfrak c$.  

\begin{lemma}
If $\mathcal D $ is a finitely compatible\label{Qstep}   set of pure
conditions   
then
there is a $\Fn( {\mathfrak c}\times\mathbb N ,2)$-name
$\dot {\mathcal D}_1$ such that $\Fn({\mathfrak c}\times\mathbb N,2)$
forces 
 that
\begin{enumerate}
 \item $\dot {\mathcal D}_1\subset\Pure$ 
is finitely directed and includes
 $\mathcal D$,
\item  $Q(\dot {\mathcal D}_1)$ is in  $\mathbb Q_{207}$,  and
\item   the ground model subsets of $\mathbb N$ is not a splitting
  family in the  further forcing extension by $Q(\dot{\mathcal D}_1)$.
\end{enumerate}
\end{lemma}

We now establish notation that will be useful when preserving
that a family of names is forced to be 
very thin.

\begin{definition} Let $Q\in \mathbb Q_{207}$ and let $\dot f$ be a
  $Q$-name such that $\Vdash_{Q} \dot f\in \mathbb N^{\mathbb N}$. A condition
  $q\in Q$ is $\dot f$-ready if, for each integer $\ell>0$, 
  $w\subset \{1,\ldots, \max(\innt(t^q_{\ell-1}))\}$, and 
  $t^q_\ell$-positive $e$, there is a $w_e\subset e$ such that
  $(w\cup w_e, \{t^q_k : k>\ell\})$
decides the value of
$\dot f(j) < \min(\innt(t^q_{\ell+1}))$
for each $j\in \{1,\ldots, \max(\innt(t^q_{\ell-1}))\} $. 
\end{definition}

\begin{lemma}
For each  $ Q\in \mathbb Q_{207}$ and 
  $  Q$-name  $\dot f$   such that $\Vdash_Q \dot f\in \mathbb
  N^{\mathbb N}$, 
  the set\label{getready} 
 of $\dot f$-ready conditions is a dense subset
  of $Q$.
\end{lemma}

\begin{proof}
  Let $Q$ and $\dot f$ be as in the statement of the Lemma and let
  $q$ be any element of $Q$. 
 For each $k\in \mathbb N$, there is a pre-dense set
  $\{ (u^k_n , T^k_n) : n\in \omega\}\subset  Q$
  satisfying that $(u_q, T^k_n)\geq q$ and
$(u^k_n, T^k_n)$ forces a 
  value on $\dot f{\restriction}\, \{1,\ldots,k\}$
  for each $n\in\omega$.  By choosing a cofinite subset
  of $T^k_n$ we may assume also that
$(u^k_n, T^k_n)$ forces that
the range of $\dot f{\restriction \,}\{1,\ldots,k\}$ is contained
in    $\min(\innt(T^k_n))$.
  Since $  Q\in 
   \mathbb Q_{207}$, there is, for each $k\in \mathbb N$, 
   a condition $
   (\emptyset , \{ t^k_\ell : \ell\in \omega\})\in 
 Q$ which is the mod finite meet of the predense
set $\{ (u^k_n, T^k_n) : n\in \omega\}$. We recall
that this means that for each $\ell\in \omega$,
 $w\subset\{1,\ldots
 \max({\mathop{int}(t^k_{\ell-1})} )\}$ and $t^k_\ell$-positive $e$,
there is a $w_e\subset e$, such that
$(w\cup w_e, \{ t^k_m : \ell<m \in \omega\})$
decides the
value of $\dot f \restriction \{1,\ldots,k\} $. Also,
if $n$ was the value  witnessing that
$(w\cup w_e, \{ t^k_m : \ell< m\in \omega\})\geq (u^k_n, T^k_n)$
then  the range of $\dot f {\restriction}\, \{1,\ldots,k\} $
is also forced to be contained in $\min(\innt(T^k_n))\leq
\min(\innt(t^k_{\ell+1}))$.

Let $T_0 = \{ t^q_\ell : \ell\in \omega\}$ and
 for each $k\in\mathbb N$,
 let $T_k = \{ t^k_\ell : \ell \in \omega\}$. 
Choose $(\emptyset, \{ t_\ell : \ell \in \omega\})$ to be
 a mod finite meet of the family
 $\{ (\emptyset,T_k) : k\in \omega \}$.
It follows easily from the definition that $(\emptyset,
 \{t_\ell : \max(u_q)<\ell\in \omega\})$ is also a mod finite meet
of this family, and so by re-indexing, we can assume
that
  $t_\ell$ is built from a finite subset of
 $T_{\max(\innt(t_{\ell-1}))}$ for each $\ell \in\omega$,
 and that $\max(u_q) < \min(\innt(t_0))$.
We check that $(u_q, \{ t_\ell : \ell\in \omega\})$ is
$\dot f$-ready. Consider  any $\ell > 0$ and
let $\bar\ell$ denote $ \max\innt(t_{\ell-1})$.
Choose any
$w\subset \{ 1, \ldots,\bar\ell \}$
and $t_{\ell}$-positive $e$.
 Choose any 
 $k$ so that  $e\cap \innt(t^{\bar\ell}_k)$ is
 $t^{\bar\ell}_k$-positive. Choose $w_e\subset (e\cap
 \innt(t^{\bar\ell}_k))$ so 
 that
 $(w\cup w_e, \{t^{\bar\ell}_m : k< m\in \omega\})$ decides
the value of $\dot f\restriction\{1,\ldots, \bar\ell\}$.
Since $(\emptyset,\{ t_m : m\in \omega\})$ is a mod finite meet of the
sequence $\{ (\emptyset, T_m ) : m\in \omega\}$, we also have
that $(\emptyset , \{ t_m : \ell < m\in \omega\}) \geq
 (\emptyset , \{ t^{\bar\ell}_m : k < m\in \omega\})$. Therefore
 $(w\cup w_e , \{t_m : \ell < m\in \omega\})\geq
  (w\cup w_e , \{ t^{\bar\ell}_m : k < m\in \omega\})$. 
This proves that $(w\cup w_e, \{ t_m : \ell < m\in \omega\})$ 
decides the value of $\dot f\restriction\{1,\ldots, \ell\}$.
   \end{proof}

\begin{corollary}
 Suppose that $\az \in \APv$ and let\label{boulderstep}
 $|P^{\az}_\kappa|\leq \theta=\theta^{\aleph_0} < \lambda$
 and let $\az_1$ denote the Cohen${}^{\theta}$-extension of $\az$.
 Then there is a $\bz\in \APv$ and a sequence
  $\{ {\mathcal D_i} : i<\kappa\}$ satisfy that, for $i<j<\kappa$,
 \begin{enumerate}
\item $\az_1 \lap 0\bz$ and $\mathcal A^{\bz}_i = \mathcal A^{\az}_i$,
\item $  {\mathcal D_i}\subset
  \wp(\mathcal L_1, P^{\az_1}_i)$,
 \item ${\mathcal D_i}\subset {\mathcal D_j}$,
 \item   
    $P^{\az_1}_i$      forces that $\{\emptyset\}\times \mathcal D_i=
\{(\emptyset ,T) : T\in 
    \mathcal D_i\}$ is a subset of $\Pure$ and
    is $\aleph_1$-directed mod finite,
\item 
  $P^{\bz}_i = P^{\az_1}_i * Q(\{\emptyset\}\times\mathcal D_i)$,
  \item $P^{\bz}_{j}$ forces that $\wp(\mathbb N, P^{\az}_i)$ is not a
    splitting family.
 \end{enumerate}
 \end{corollary}

 \begin{proof}
   Let $\az\in \APv$ and
   $\kappa\leq |P^{\az}_\kappa|^{\aleph_0}\leq\theta<\lambda$. If
   needed, we can 
   first extend $\az$ by applying Lemma
   \ref{cohen}, so as to assume that $\Vdash_{P^{\az}_i } \mathfrak c
   = \theta$ for all $i<\kappa$.  Before proceeding we note that
it follows from Proposition \ref{cohen} and Lemma \ref{q207}, that
condition 
(4) will imply that
 $\{ P^{\bz}_i : i< \kappa\}$ is a continuous $\cprec$-chain of ccc
 posets as required in the definition of $\APv$.   
We construct the sequence
     $\{ \mathcal D_i : i<\kappa\}$ by recursion  on $i<\kappa$. 
 It follows from
    Lemma \ref{Qstep} that there is a set
     $\mathcal D_0\subset \wp(\mathcal L_1,P^{\az}_0*\Fn(\{0\}\times\theta
\times \mathbb
N,2))$  such that $\{\emptyset\}\times \mathcal D_0$ is forced
to be a subset
     of $\Pure$ that is $\aleph_1$-directed mod finite. Assume that
     $\bar\i<\kappa$ and that $\{ \mathcal D_i: i<\bar\i\}$ has been
     constructed 
so that for $i<j<\bar\i$,  properties (2)-(6) hold  and so that
$P^{\bz}_{i+1}$ forces that $\mathcal A^{\az}_i$ is very thin over
the forcing extension by $P^{\bz}_i$. It will be most convenient to
continue the argument in a forcing extension.

Let $G_\kappa\subset P^{\az}_\kappa$ be a generic filter and,
 for each $i<\kappa$, let $G_i = G_\kappa\cap P^{\az}_i$. For each
 $i<\kappa$, let $H_i\subset \Fn(i{+}1\times\theta\times\mathbb N,2)$
 be a filter so that $G_i\times H_i$ is a generic filter for
 $P^{\az}_i*\Fn(i{+}1\times\theta\times\mathbb N,2)$. It follows
 that $\bar G = 
G_{\bar \i}*\left(\bigcup \{H_i : i<\bar\i\}\right)$ is a
 generic filter for $P^{\az}_{\bar\i}*\Fn(\bar\i\times
 \theta\times\mathbb N,2)$.  We work in the forcing extension $V[\bar
 G]$. We first handle the case when $\bar\i$ is a limit ordinal.  By
 the definition of the family $\APv$,  the sequence
 $\{ \mathcal A^{\az}_i : i<\kappa\}$ is not a concern,
 as in Definition \ref{apv}, when defining
 $P^{\bz}_{\bar\i}$ in the limit case.  
 It
 should be clear that the $\bar G$-interpretation of the 
collection   $\mathcal E = \bigcup \{ \{\emptyset\}
\times \mathcal D_i : i< \bar\i\}$ is a finitely directed subset of
$\Pure$. We proceed as in the base case. By Lemma \ref{Qstep}, there
is a $\Fn(\{\bar\i\}\times\theta\times\mathbb N,2)$-name,
$\dot{\mathcal E}'$,
 of
a subset of $\Pure$
that is forced to be an $\aleph_1$-directed mod finite extension of
$\mathcal E$  that further forcing by $Q(\dot{\mathcal E}')$
ensures that the family $[\mathbb N]^{\aleph_0}\cap
V[\bar G]$  is not a splitting family.  The family
 $\mathcal D_{\bar\i}$ is a subset of $\wp(
\mathcal L_1,P^{\az}_{\bar\i}*\Fn(\bar\i{+}1\times\theta
\times \mathbb N,2))$ that contains $\mathcal D_i$ for all $i<\bar\i$
and is forced to satisfy that $\dot{\mathcal E}'$ is equal
to $\{ \{\emptyset\}\times T : T\in \mathcal D_{\bar\i}\}$.

Now we may assume that $\bar\i = i+1$ and we note that $\mathcal
A^{\az}_i$ is a family of $P^{\az}_{i+1}$-names that is forced to be
very thin over the forcing extension by $P^{\az}_i$. It follows from
Lemma \ref{cohen} that $\mathcal A^{\az}_i$ is forced to be very thin
over the model $V[G_i*H_i]$. We again work in the forcing extension
$V[\bar G]$ where $\bar G = G_{i+1} *H_i$.  Let $\mathcal A$ denote
the ideal generated by the $\bar G$ interpretations of the names from
 $\mathcal A^{\az}_i$. 
Let $\mathcal E$ denote the
$\aleph_1$-directed mod finite family
  $\{ \{\emptyset\}\times T : T \in \mathcal D_i\}$.  
Let $\dot x_0$ denote the canonical subset
   of $\mathbb N$ added by $\Fn(\{(i+1,0)\}\times \mathbb N,2)$ over
   the model $V[\bar G]$ as in Definition \ref{cohsequence}. Let $\{
\dot  n_m : m\in \omega\}$ denote the name of the increasing
enumeration of $\dot x_0$. 
For all $a\in \mathcal A$, let
$\dot I(a)$  be a canonical $\Fn(\{(i+1,0)\}\times \mathbb
N,2)$-name for the set $\{ m\in \mathbb N : a\cap [\dot n_{m},
 \dot n_{m+1}] =\emptyset \}$.  For each $q\in Q(\mathcal E)$,
let $\dot J(a,q)$ be a canonical $\Fn(\{(i+1,0)\}\times \mathbb
N,2)$-name for the set 
$ \{ \ell\in \mathbb N:
(\exists m \in \dot I(a))~~
\dot n_m \leq \max(\innt(t_{\ell-1})) < \min(\innt(t_{\ell+1}))
\leq \dot n_{m+1}\}$.

\begin{claim}
  The family $\mathcal E_0 = 
\{ (\emptyset, \{t^q_\ell : \ell\in \dot J(a,q)\}) :
  a\in \mathcal A, q\in Q(\mathcal E)\}$ is forced to be a
finitely directed subset of $\Pure$.
\end{claim}

\bgroup

\def\proofname{Proof of Claim:~}

\begin{proof}
Each $q\in \mathcal E$ is in the model $V[G_i*H_i]$
and each $a\in \mathcal A$ is thin over that model. Therefore
there is an infinite set of $\ell\in\mathbb N$ such
that $a$ is 
disjoint from $[\max(\innt(t_{\ell-1})), \min(\innt(t_{\ell+1}))]$.
  This implies, by a simple genericity argument,
that  $\dot J(a,q)$ is 
forced to be an infinite set
for each $a\in \mathcal A$ and $q\in \mathcal E$.
Let $H$ be the 
generic filter for $\Fn(\{(i+1,0)\}\times \mathbb N,2)$
that is equal to $H_{i+1}\cap \Fn(\{(i+1,0)\}\times \mathbb N,2)$.
For all $a\in \mathcal A$ and $q\in \mathcal E$, let
$I(a)$ and $J(a,q)$ denote the interpretations by $H$
of
$\dot I(a)$ and $\dot J(a,q)$ respectively. Similarly
let $\{ n_m : m\in \mathbb N\}$ denote the increasing enumeration of
the interpretation of $\dot x_0$.

Evidently, if $a_1\subset a_2$ are elements of $\mathcal A$,
 then $I(a_2)$ is  a subset of $ I(a_1)$.
 To prove the claim it suffices to assume that if $q_2\geq q_1$ are in
 $\mathcal E$, then, for each $\ell\in  J(a,q_2)$,
 $t^{q_2}_\ell$ is built from a finite subset of
 $\{ t^{q_1}_k : k\in  J(a,q_1)\}$.  Fix any
  $\ell \in J(a,q_2)$ and choose  minimal finite subsets
  $B_{\ell-1},B_\ell$ and $B_{\ell+1}$  of $ \omega$ 
  such that, for each $r\in\{-1,0,1\}$, 
  $t^{q_2}_{\ell+r}$ is  built from
  $\{ t^{q_1}_k : k\in B_{\ell+r}\}$.  Let $k_{0}$ be the
  maximum element of $B_{\ell-1}$ and let $k_1$ be the minimum element
  of $B_{\ell+1}$. From the definition of $\mathcal Q_{Bould}$,
  we have, for each $k\in B_\ell$, 
  that
  \begin{enumerate}
\item  $k_{0}< k < k_{1}$,
\item  $\max(\innt(t^{q_2}_{\ell-1}))
  \leq \max(\innt(t^{q_1}_{k_{0}}))\leq
  \max(\innt(t^{q_1}_{k-1}))  $, and
  \item $  \min(\innt(t^{q_1}_{k+1}))
\leq \min(\innt(t^{q_1}_{k_{1}}))
    \leq \min(\innt(t^{q_2}_{\ell-1}))
$.
  \end{enumerate}
  Fix the unique $m\in I(a)$ such
  that $n_m\leq \max(\innt(t^{q_2}_{\ell-1})) <
  \max(\innt(t^{q_2}_{\ell+1})) <   n_{m+1}$, and now conclude
  that
  $n_m\leq \max(\innt(t^{q_1}_{k-1})) <
  \max(\innt(t^{q_1}_{k+1})) <   n_{m+1}$. This proves
  that $B_\ell\subset J(a,q_1)$ as required.
\end{proof}

Let
$\widetilde H$ be the generic filter
 $H_{i+1}\cap \Fn(\{i+1\}\times \theta\times\mathbb N,2)$.
For each $q\in Q(\mathcal E)$ and $a\in \mathcal A$,
let $q(a)\in Q({\mathcal E}_0)$ denote the condition
$(u_q, \{ t^q_\ell : \ell\in J(a,q)\})$.  

\begin{claim} In the 
  forcing extension  $V[\bar G*\widetilde H]$
there is a family $\mathcal E_1\subset \Pure$ such that
\begin{enumerate}
\item $\mathcal E_0\cup \mathcal E$ is a subset of $ \mathcal E_1$,
\item $\mathcal E_1$ is $\aleph_1$-directed mod finite,
\item
     the family $[\mathbb N]^{\aleph_0}\cap V[\bar G]$ is not a
     splitting family
      in the further forcing extension by $Q(\mathcal E_1)$,
    \item for each $Q(\mathcal E)$-name $\dot f\in V[G_i*H_i]$
 of an element of
       $\mathbb N^{\mathbb N}$ and 
       each $\dot f$-ready  $q\in Q(\mathcal E)$,
       $q(a)\Vdash_{Q({\mathcal E}_1)} (\exists n)~~ a\cap
       [n, \dot f(n)] = \emptyset$ for
       each $a\in \mathcal A$.
\end{enumerate}
  \end{claim}

  \begin{proof}  We simply apply Lemma \ref{Qstep} to select
    $\mathcal E_1$.  This ensures that conditions (1)-(3) hold.
    Now we verify that (4) holds. Let $\dot f$, $a$ and $q$ be
    as in the statement of (4). Let $r\in Q(\mathcal E_1)$ be any
    condition stronger than $q(a)$. Fix any $k$ so that
     $\max(u_r) < \min(\innt(t^r_k))$. Since $r\geq q(a)$, there is a
     finite subset $B$ of $J(a,q)$ such that $t^r_k$ is built from
     $\{ t^q_\ell : \ell \in B\}$. Choose any $\ell\in B$ such
     that $e = \innt(t^r_k)\cap \innt(t^q_\ell)$ is
     $t^q_\ell$-positive. Since $\ell\in J(a,q)$, there is an 
$m\in I(a)$ so
     that
      $n_m \leq \max(\innt(t^q_{\ell-1})) < \min(\innt(t^q_{\ell+1}))
      < n_{m+1}$. Since $q$ is $\dot f$-ready,  there is a $w_e\subset e$
      such that $(u_r\cup w_e, \{t^q_j : \ell<j\in \omega\})$ forces
      that $\dot f(n_m) < \min(\innt(t^q_{\ell+1}))$.
     Since $m\in I(a)$, this completes
      the proof that $r$ has an extension forcing
      that $a$ is disjoint from $[n_m, \dot f(n_m)]$. 
  \end{proof}

  \egroup

  The proof of the Corollary is completed by choosing a
  subset
  $\mathcal D_{\bar\i}$  of $ \wp(\mathcal L_1,
   P^{\az}_{\bar\i}* \Fn(\bar\i+1 \times \theta\times\mathbb N,2))$
   (recall that $\bar\i = i+1$) so that $\mathcal D_i\subset
   \mathcal D_{\bar\i}$ and
   $\{ \{\emptyset\}\times T : T\in \mathcal D_{\bar\i} \}$ is forced
   to equal $\mathcal E_1$.  We prove that $P^{\bz}_{i+1}$ forces
   that $\mathcal A_i$ is very thin over the forcing extension
   by $P^{\bz}_{i}$. Recall that $G_{i+1}$ is a
   $P^{\az}_{i+1}$-generic filter and, similarly,
   $G_i = G_{i+1}\cap P^{\az}_i$ is $P^{\az}_i$-generic.
   Let $\mathcal E_i$ denote the interpretation of
   $\{\emptyset\}\times \mathcal D_i$, and similarly,
   let $\mathcal E_{i+1}$ denote the interpretation of
   $\{\emptyset\}\times \mathcal D_{i+1}$.
    We already know
   that $\mathcal A_i$ is forced to be very thin over the model
    $V[G_i*H_i]$, so it suffices to consider a 
$Q(    \mathcal E_i)$-name $\dot f$ in $V[G_i*H_i]$ of an element of
     $\mathbb N^{\mathbb N}$. By Lemma     \ref{getready}, the set of
     $\dot f$-ready conditions is a dense subset of
      $Q(\mathcal E_i)$. Since $\mathcal E_i$ is $\aleph_1$-directed
      mod finite, we also have, by Proposition \ref{q207}, that
      the set of $\dot f$-ready conditions from $Q(\mathcal E_i)$
      is a pre-dense subset of $Q(\mathcal E_{i+1})$. The result now
      follows from item (4) of Claim 3. 
 \end{proof}

\begin{definition}
 For any $\az\in \APv$ we will say\label{boulderdef}
  that $\bz\in \APv$ is a $Q(\vec{\mathcal D})$-extension
  of $\az$ if there is  cardinal $\theta<\lambda$ with
   $|P^{\az}_\kappa| \leq \theta = \theta^{\aleph_0}$ and
  a sequence
  $\{ \mathcal D_i : i<\kappa\}$ such that for all $i<j<\kappa$, 
 
\begin{enumerate}
\item  $\az_1 \leq^0_{\AP} \bz$ where $\az_1$ is the
  Cohen${}^\theta$-extension of $\az$,
  \item $\mathcal A^{\az}_i = \mathcal A^{\bz}_i$,
\item $\mathcal D_i \subset \wp(\mathcal L_1,P^{\az_1}_i)$,
 \item $\{\emptyset\}\times 
   \mathcal D_i$ is forced by $P^{\az_1}_i$
 to be an
   $\aleph_1$-directed subset of $\Pure$,
 \item $\mathcal D_i \subset \mathcal D_j$, 
 \item $P^{\bz}_i = P^{\az_1}_i * Q (\{\emptyset\}\times\mathcal
   D_i)$,
   \item $P^{\bz}_{j}$ forces that $\wp(\mathbb N, P^{\az}_i)$ is not
     a splitting family.   
\end{enumerate}
\end{definition}

Now we formulate the $\APv$ version of Theorem \ref{mainAP}.

\section{$\pi p(\mathfrak U) \leq \kappa =
\mathfrak b $ and 
$ \mathfrak s = \lambda$}

Fix, as in Section 4, a
1-to-1 function $h$ from $\lambda$ onto  $H(\lambda)$.
Recall that by our assumption
Hyp$(\kappa,\lambda)$, 
 $E$ is a stationary subset
 of   $  S^{\lambda}_\kappa$
 and 
$\{ C_\alpha : \alpha \in \lambda\}$
is  $\square$-sequence. 
We let  $\{ X_\alpha : \alpha \in E\}$
be the
$\diamondsuit$-sequence on $\lambda$
as postulated in by Hyp$(\kappa,\lambda)$.

\begin{theorem}
Assume Hyp$(\kappa,\lambda)$.
 There\label{mainAPv} is a sequence $\{ \az_\alpha , \zeta_\alpha :
 \alpha\in \lambda\}$ 
 such that for each limit $\delta\in \lambda$:
 
\begin{enumerate}
\item the sequence $\{ \az_\alpha :\alpha < \delta\}$ is  $\lap
  *$-increasing
  subset of $\APv$,
  \item   $\{ \zeta_\alpha : \alpha <\delta\}\subset\lambda$ is
    non-decreasing, 
and   $\zeta_\delta \in \lambda$ is the supremum,
 \item if $\delta\notin E$, 
 the sequence $\{ \az_\alpha : \alpha\in \acc(C_\delta) \cup \{\delta\}\}$ is
  a
    $\lap 0$-increasing continuous chain,
    \item if $\delta\in E$ and\label{iultrav}
     $\mathcal E_\delta = \{ h(\xi) : \xi\in X_\delta\}$ is
    a maximal 
   subset of $\wp(\mathbb N,  P^{\az_\delta}_\kappa ) $ that is forced
   by $P^{\az_\delta}_\kappa$ to be a free
   ultrafilter on $\mathbb N$, then $\mathcal E_\delta\cap 
    \mathcal A^{\az_\delta}_i$ is not empty for all $i<\kappa$,
    \item $\az_{\delta+1}$ is the Cohen$^{\omega_1}$-extension of $\az_{\delta}$
    and $\zeta_{\delta+1} = \zeta_\delta$,

    \item if $\alpha = \delta+1$ then
     $\zeta_{\alpha+1} = \zeta_{\alpha}$ 
   and $\az_{\alpha+1} $ is a
      Cohen$^{\theta_\alpha}$-extension of $\az_{\alpha}$ where
   $\theta_\alpha= |P^{\az_{\alpha}}_\kappa|^{\aleph_0}$
      \item if $\alpha = \delta+2$, then $\zeta_{\alpha+1}=\zeta_\alpha$
      and $\az_{\alpha+1}$ is a\label{iboulder} 
      $Q(\vec{\mathcal D})$-extension of $\az_{\delta}$,
      \item if $\alpha \in (\delta+2,\delta+\omega)$, then 
       $\zeta_{\alpha+1}$ is the minimal value strictly above $\zeta_\alpha$
       such that $\dot Q_{\alpha+1} = h(\zeta_{\alpha+1}-1)$
      has cardinality less than $\kappa$ and\label{ivccc}   is a 
         $P^{\az_{\alpha}}_\kappa$-name of a poset that is forced to be ccc,
         and 
      $\az_{\alpha+1} = \az_{\alpha}*\dot Q_{\alpha+1}$ as in Definition
       \ref{cccstep}. 
\end{enumerate}
\end{theorem}

\begin{proof}
The proof only involves very minor modifications to the proof of
Theorem \ref{mainAP} and can be omitted.
\end{proof}

\begin{theorem} 
  Assume Hyp$(\kappa,\lambda)$. There is a ccc poset $P$ forcing
  that $\mathfrak s = \mathfrak b = \lambda$, $MA(\kappa)$, and
  $\pi p(\mathcal U) \leq \kappa$ for all free ultrafilters
  $\mathcal U$ on $\mathbb N$.
\end{theorem}

\begin{proof}
 Let $\{ \az_\alpha , \zeta_\alpha : \alpha \in \lambda\}$ be the sequence 
 constructed in Theorem \ref{mainAP}. Let $P$ be the poset
  $\bigcup \{ P^{\az_\alpha}_\kappa : \alpha \in \lambda\}$. Since
  $\{ P^{\az_\alpha}_\kappa : \alpha \in \lambda\}$ is a strongly continuous
   $\cprec$-increasing chain of ccc posets, it follows that $P$ is ccc.
   Furthermore $\wp(\mathbb N , P)$ is equal to the union of the
   increasing sequence $\{\wp(\mathbb N, P^{\az_\alpha}_\kappa)
   :\alpha < \lambda\} $. It then follows immediately from condition
   (\ref{iboulder})
  that $\mathfrak s$ is forced to be $\lambda$. By Proposition \ref{bkappa}, 
$P$ forces that   $\mathfrak b \leq \kappa$. The fact 
that $P$ forces that $\mathfrak b = \kappa$ follows once we
 we note  that $P$ forces that $MA(\kappa)$ holds.
This is proven exactly as in the proof of Theorem \ref{b=s=lambda}
and so can be omitted.
Similarly, simply repeating that portion of 
the proof from Theorem \ref{b=s=lambda}
also proves
that $P$ forces
that $\pi p(\mathcal U)\leq\kappa$ for all free
ultrafilter $\mathcal U$ on $\mathbb N$.
\end{proof}

\end{document}